\documentclass[11pt]{article}
\usepackage{amsmath, amssymb, amsthm, verbatim,enumerate,bbm}
\usepackage{indentfirst}

\date{\today}
\allowdisplaybreaks

\theoremstyle{plain}
\newtheorem{theorem}{Theorem}[section]
\newtheorem{lemma}[theorem]{Lemma}
\newtheorem{claim}[theorem]{Claim}

\newtheorem{corollary}[theorem]{Corollary}

\newtheorem{remark}[theorem]{Remark}
\newtheorem{definition}[theorem]{Definition}

\newcommand{\BN}{\mathbb N }

\title{ILL}

\title{Law of the Iterated Logarithm for random graphs}

\author{ Asaf Ferber
\thanks{Department of Mathematics, Yale University. Email:
asaf.ferber@yale.edu.} \and Daniel Montealegre \thanks{Department of Mathematics, Yale University. Email: daniel.montealegre@yale.edu}\and Van Vu \thanks{Department of
Mathematics, Yale University. Email: van.vu@yale.edu; V. Vu's research is supported by NSF grant   DMS-1500944 and AFORS grant FA9550-12-1-0083} }

\begin{document}

\maketitle
\begin{abstract}

A milestone in Probability Theory is the \emph{law of the iterated logarithm} (LIL), proved by Khinchin and independently by Kolmogorov in the 1920s, which asserts that for iid random variables $\{t_i\}_{i=1}^{\infty}$ with mean $0$ and variance $1$

$$ \Pr \left[ \limsup_{n\rightarrow \infty} \frac{ \sum_{i=1}^n t_i }{\sigma_n \sqrt {2 \log \log n }} =1 \right] =1 . $$

 In this paper we prove that LIL holds for various
 functionals of random graphs and hypergraphs models.
 We first prove LIL for  the number of copies of a fixed  subgraph $H$.  Two harder results concern the number of
 global objects: perfect matchings and Hamiltonian cycles.  The  main new ingredient in these results  is  a   large deviation bound, which may be of independent interest.  For random $k$-uniform  hypergraphs, we obtain the Central Limit Theorem (CLT)  and LIL for the number of Hamilton cycles.

 \end{abstract}

\section{Introduction}

Let $\{t_i\}_{i=1}^{\infty}$ be an infinite sequence of iid random variables with mean $0$ and variance $1$. Two key results in probability theory are the central limit theorem and the law of the iterated logarithm. The \emph{central limit theorem}
(CLT) states that for $X_n:= \sum_{i=1}^n t_i$, one has

$$ \frac{ X_n  }{\sigma_n  }  \longrightarrow N(0,1),$$ where $\sigma_n := \sqrt {Var X_n }  =\sqrt n $ and $N(0,1)$ denotes the standard gaussian distribution. The \emph{law of the iterated logarithm} (LIL), proved by Khinchin \cite{Khinchin} and Kolmogorov \cite{Kolmogorov}, asserts that

$$ \Pr \left[ \limsup_{n\rightarrow \infty} \frac{ X_n  }{\sigma_n \sqrt {2 \log \log n }} =1 \right] =1 . $$

\noindent The $\log \log n$ term reveals  a subtle
correlation between the $X_i$'s, especially those with indices close to each other.

The theory of random graphs (hypergraphs) contains several central limit theorems, some of which are among the most well known
results in the field. It is natural to wonder if the LIL also holds.
The goal of this paper is to initiate this investigation  and provide
the first few rigorous results. To our surprise,  this  natural problem has not been   studied before and we hope this paper will motivate further activity.

Let $p$ be a fixed constant in $(0,1)$. We consider the infinite random hypergraph $H^k(\BN, p)$ on the vertex set $\BN$ where we add every $k$-subset $S\subseteq \BN$ as an edge with probability $p$ independently. This gives rise to a nested sequence of
  random hypergraphs where $H^k(n,p)$ is defined by restriction to the first $n$ vertices $[n]:=\{1, \dots, n \}$. The atom iid variables are $t_{S}$ which represent the edges ($t_{S} =1$ if $S$ forms an edge and $0$ otherwise). In the case of graphs (that is $k=2$) we denote $H^k(\BN,p)$ by $G(\BN,p)$ and $H^k(n,p)$ by $G(n,p)$. In this way, we obtain the usual binomial random graph model. We also consider the infinite random bipartite graph $B(\BN,p)$ on vertex set $A\cup B$, where $A$ and $B$ are two disjoint copies of $\BN$, and every pair $ab\in A\times B$ forms an edge with probability $p$, independently. Let $B(n,p)$ be obtained from $B(\BN,p)$ by restricting $A$ and $B$ to their first $n$ elements.

Many CLT's in the theory of random graphs involve
some sort of counting functions.  For instance, counting
the number of copies of a fixed graph (such as triangles or $C_4$'s)
is a classical problem; see \cite{KarR, Kar, Rucinski} and the references therein (the interested reader can also find a detailed discussion in \cite{JansonLuczakRucinski}, Chapter 6). In this case, the question of when the CLT holds is well understood.

\begin{theorem} Fix a nonempty graph $G$, and let $X_n$ count the number of copies of $G$ in $G(n,p)$. Let $m(G)=\max\{|E(H)|/|V(H)| : H\subset G\}$. If $p=p(n)$ is such that $np^{m(G)}\rightarrow \infty$ and $n^2(1-p)\rightarrow \infty$, then $(X_n-\mathbb{E}[X_n])/\sqrt{Var(X_n)}$ tends in distribution to $N(0,1)$.
\end{theorem}

It is more challenging to count global objects.
 In \cite{Janson} Janson considered the numbers of spanning trees, perfect matchings and Hamilton cycles in random graphs. He  showed  these counting functions are
  log-normal for $G(n,p)$  in certain ranges of density.  Results of a similar flavor (and shorter proofs) were also  obtained later by Gao \cite{Gao}.
\begin{theorem}Let $X_n$ be the random variable that counts number of spanning trees, perfect matchings, or Hamilton cycles in $G(n,p)$. Fix a constant $p<1$. Let $p(n)\rightarrow p$. If $\liminf n^{1/2} p(n)>0$, then
$$
p(n)^{1/2}\left(\log(X_n)-\log(\mathbb{E}[X_n])+\frac{1-p(n)}{c p(n)}\right)\rightarrow N\left(0,\frac{2(1-p)}{c}\right)
$$where $c=1$ in the case of spanning trees and Hamilton cycles, and $c=4$ in the case of perfect matchings.
\end{theorem}

Throughout this paper, we use $X_n$ to denote a statistic of the random model under consideration (that is, $H^k(n,p)$ or $B(n,p)$), with mean $\mu_n$ and variance $\sigma_n^2$, which may vary in each occasion.
First, we consider the case $X_n$ is the number of copies of a fixed graph $H$ in $G(n,p)$ and prove

\begin{theorem}\label{main1} For a fixed graph $H$, let $X_n$ denote the number of copies of $H$ in $G(n,p)$.
The sequence $X_n$ satisfies the  LIL, namely

$$\Pr \left[\limsup_{n\rightarrow \infty}
\frac{X_n  -\mu_n }{\sigma_n \sqrt{2\log\log n}}=1\right]=1.$$
\end{theorem}

The key ingredient in the
 proof of Theorem \ref{main1} is  to overcome the fact that the terms in $X_n$ are not completely independent.

Second, we consider the case where $X_n$ is the number of perfect matchings in $B(n,p)$. In this case, we obtain a LIL for the random variable $\log X_n$.

\begin{theorem}\label{thm:main}Let $X_n$ be the  number of perfect matchings in  $B(n,p)$ and set $Y_n := \log X_n$. Then the sequence $Y_n$ satisfies the LIL, namely
\begin{align}\label{eqtn main thm}
\Pr  \left[\lim\sup_{n\rightarrow \infty}\frac{ Y_n- \log (n! p^n ) + \frac{1-p}{2p}  }{ \sqrt{2\log \log n}  \sqrt { \frac{1-p}{p}} }=1\right]=1
\end{align}
\end{theorem}

Third, we consider  the number of Hamilton cycles in $G(n,p)$ and prove

\begin{theorem}
  \label{thm:main2} Let $X_n$ be the number of Hamiltonian cycles in $G(n,p)$ and set $Y_n := \log X_n$.
  The sequence $Y_n$ satisfies the LIL, namely
$$\Pr \left[\limsup_{n\rightarrow \infty}
\frac{Y_n  -  \log \left(\frac{(n-1)!}{2}p^n\right)+\frac{1-p}{p} }{\sqrt{\frac{2(1-p)}{p}}\sqrt{2\log\log {n}}}=1\right]=1.$$
\end{theorem}

The proofs of the last two theorems are more involved.
Our new key ingredient is a large deviation bound
on $X_n$ (the number of perfect matchings or Hamiltonian cycles, respectively), which  appears to be new and could be of independent interest.

\begin{remark} Note that we did not write $(Y_n-\mathbb{E}[Y_n])/\sqrt{Var(Y_n)}$ in theorems \ref{thm:main} and \ref{thm:main2},  as the expected value and variance of $Y_n$ are unknown.  We conjecture that
	constants used in the theorem are good approximations of these quantities.
\end{remark}

Next, we consider the case of $k$-uniform random hypergraphs. In this setting, the CLT and the LIL for the number of copies of a fixed subhypergraph can be obtained in a similar way to the graph case. Therefore, we  focus on global  structures,  Hamiltonian cycles in particular.

To start, there are many ways to define a  cycle in a hypergraph.  We work with the following: an $\ell$-overlapping Hamilton cycle is a cyclic ordering of the vertices $v_1,\ldots,v_n$
for which the edges consisting  of $k$ consecutive vertices and two
consecutive edges overlap in exactly $\ell$ vertices. The case $\ell=1$ is known as a ``loose Hamilton cycle" and the case $\ell=k-1$ is known as  ``tight Hamilton cycle" (note that the case $\ell=0$ corresponds to a perfect matchings).
 Our next result works for all $\ell$, but for the sake  of presentation we state it for loose Hamilton cycles (which from now on will be referred to as Hamilton cycles).

Let $X_{n}(k) $ denote the number of Hamilton cycles in $H^k(n,p)$  with mean $\mu_{n}(k)$ and variance $\sigma_{n}(k) ^2$.
We have found out, somewhat surprisingly, that   for $k \ge 3$, $X_n(k) $  themselves satisfy
the CLT, as opposed to the case $k=2$ where $\log X_n(2)$ satisfies the CLT. The reason  lies in the fact that unlike the case $k=2$, for $k\geq 3$, if we choose a few Hamilton cycles at random, it is very unlikely for them to have common edges and therefore the variance
of the counting function is much smaller compared to $\mu_n(k)^k$.
A similar  observation has been used by Dudek and Frieze in \cite{DF} and \cite{DF1} where they determined the threshold behavior of $\ell$ Hamilton cycles.

\begin{theorem}\label{main:CLT}
For any $ k \ge 3$, the  sequence $ X_{n}(k) $ satisfies the CLT, namely

$$\frac{ X_{n} (k)  -  \mu_{n} (k) } {\sigma_{n} (k)   } \longrightarrow N(0,1). $$

\end{theorem}

Finally, we  show that  for $k \ge 4$, the sequence $X_{n} (k) $  satisfies a LIL.

\begin{theorem}\label{main:ILL}
For $k \ge 4$,  the sequence $X_{n} (k) $ satisfies the LIL, namely

$$\Pr \left[\limsup_{n\rightarrow \infty} \frac{X_{n} (k) - \mu_{n} (k) }{\sigma_{n} (k) \sqrt {2\log\log n}}=1\right]=1.$$
\end{theorem}

We conclude  this section with a few  remarks. First, there are many other CLTs in the random  graphs/hypergraphs   literature, and
it is natural to raise the validity  of the LIL in each situation. We hope that this paper will motivate further research in this direction.

As far as the new results are concerned, we prove them under the condition that $p$ is a fixed constant in $(0,1)$.  Since we work with
a random infinite graph, letting $p$ depend on $n$ (as one usually does for $G(n,p)$) does not make sense.
However, one can still consider the sparse case by modifying the definition.  For instance, one can  say that  the edge  $ij\in \mathbb{N}^2$ appears with probability $p(\max\{i,j\})$, independently,
where  $p(k)$ is a sequence of positive numbers  tending to $0$ with $k$. It is an interesting question to determine  those ranges of densities  for which LIL holds.

For  a  technical reason, the proof of Theorem \ref{main:ILL} requires $k \ge 4$. We leave the case $k=3$ as an open problem.

\vskip2mm

\noindent  {\bf Notation.} Throughout the paper, we assume that $n$ is sufficiently large, whenever needed. All asymptotic notation is used under the assumption that
$n \rightarrow \infty$. We will be using the following notation through the paper:
\begin{itemize}
\item
$K_n$ the complete graph on the vertex set  $[n]=\{1,\ldots,n\}$.
\item $(t)_\ell:=t(t-1)\ldots(t-\ell+1)$.
\item $G(n,m)$ is the random graph chosen uniformly at random from the set of all graphs on vertex set $[n]$ with exactly $m$ edges.
\item $B(n,m)$ is the random graph chosen uniformly at random from the set of all bipartite graphs, with vertex sets of sizes $n$ with exactly $m$ edges.
\item For a random variable $X$, we write $X^*$ for its normalization: $X^*:= (X-\mathbb{E}[X])/(\sqrt{Var(X)})$.
\item For a graph $H$ we define $\mathcal H$ to be the set of all (labeled) copies of $H$ in
the infinite complete graph on vertex set $\mathbb{N}$. For each
$n\in \mathbb{N}$, we define $\mathcal H_n$ to be the subset of
$\mathcal H$, consisting of all copies of $H$ in $K_n$ (that is, all graphs in $\mathcal H$ which are contained in $[n]$).
\item  Given  a
copy $h\in \mathcal H$, we denote by $V(h)$ and $E(h)$ its vertex
set and edge set, respectively.
\item In the special case where $H$ is a
\emph{triangle} (that is, a graph on $3$ vertices $\{x,y,z\}$ where
all the three possible edges $\{xy,yz,zx\}$ appear), we replace
$\mathcal H$ with $\mathcal T$ in all of the previous notation.
\item We assume that an enumeration $\mathcal
H=\{h_1,h_2,\ldots\}$ is fixed so that for every $n\in \mathbb{N}$
we have $\mathcal H_n=\{h_1,\ldots,h_\ell\}$, where $\ell$ is the
number of labeled copies of $H$ in $K_n$. Note that such an
enumeration can be easily obtained by an induction on $n$.
\item Suppose $G$ is a random graph (taken from any arbitrary distribution). To each copy $h\in \mathcal H$, we associate an indicator random variable $\xi^G_h $. Whenever the model $G$ is clear from the context, we simply write $\xi_h$.
\item For a collection $\mathcal S$  of copies of $H $ we have $X_ {\mathcal S}  :=\sum_{h\in
\mathcal S}\xi_h $.
\item Let $\Phi(x)$ denote the cumulative distribution function of the standard gaussian $N(0,1)$:
$$\Phi(x):= \Pr[ N(0,1) \le x] = \frac{1}{\sqrt{2\pi}}\int_{-\infty}^xe^{-t^2/2}dt. $$
\item For an event $\mathcal E$, we denote its complement by $\neg \mathcal E$ (i.e., the event that $\mathcal E$ does not hold).
\end{itemize}

{\bf Organization of the paper.} The rest of the paper is organized
as follows. In Section \ref{sec:tools} we  collect the  tools which
are used for the proof of our main results. In Sections
\ref{sec:upper} and \ref{sec:lower} we prove the upper and lower
bounds for Theorem \ref{main1}, and in Section \ref{sec:calculations}
we explain some inequalities we use during the proof. In Section \ref{perfect matchings} we prove Theorem \ref{thm:main}, and in Section  \ref{sec: main2} we prove Theorem \ref{thm:main2}. Both of these sections are split into two subsections containing the proof of the upper bound and the lower bound, respectively. Section \ref{sec:clt hypergraphs} contains  the proof of \ref{main:CLT}, and Section \ref{sec: ill hypergraph} contains the proof of \ref{main:ILL}. Section \ref{upper tail estimates} contains the new large deviation estimates we need on  perfect matchings and Hamilton cycles. The appendix contains some rather routine, but tedious, calculations and approximations that we use throughout the paper.

\section{Tools}\label{sec:tools}

In this section we introduce the main tools to be used in the proofs
of our results.
%We need to employ standard bounds on large deviations of random
%variables. We mostly use the following well-known bound on the lower
%and the upper tails of the Binomial distribution due to Chernoff
%(see e.g. \cite{AlonSpencer}, \cite{JansonLuczakRucinski}).
%
%\begin{lemma}\label{Che}
%Let $X \sim \emph{\text{Bin}}(n,p)$, let $\mu=\mathbb{E}(X)$ and let
%$a>0$. Then
%\begin{itemize}
%    \item $\Pr\left[X\le\mu-a\right]<e^{-\frac{a^2}{2\mu}}$;
%    \item $\Pr\left[X\ge\mu+a\right]<e^{-\frac{a^2}{2(\mu+a/3)}}$.
%\end{itemize}
%\end{lemma}
As a first tool, we present Janson's inequality (see e.g.\
\cite{JansonLuczakRucinski}, Theorem 2.14), which will be used in
order to get lower tail estimates for the number of copies of a
fixed graph $H$ in certain random graphs. We only use it in the model $G(\mathbb N,p)$ where $p$ is a fixed constant. For the convenience of the
reader, we state the inequality tailored for our use later (with respect to the
$\xi^G_h$'s which were previously defined). Before doing so, we need
some notation. Let $m\leq n$ be two positive integers, and let
$\mathcal S:=\mathcal H_n\setminus \mathcal H_m$. Consider the
random variable $\xi^G_{\mathcal S}=\sum_{h\in \mathcal S}\xi^G_h$, let $\mu_{\mathcal S}$ be its
expectation, and let $$\Delta:=\sum_{h,h'\in \mathcal S \text{ s.t }\\ h\cap h'\neq \emptyset}
\mathbb{E}[\xi_h\xi_{h'}].$$ With this notation in hand we are ready to state
the theorem.

\begin{theorem}\label{Janson}
For a fixed graph $H$ and for every $0\leq t\leq \mu_{\mathcal S}$
we have
$$\Pr\left[\xi^G_{\mathcal S}\leq \mu_{\mathcal S}-t\right]\leq
e^{-\frac{t^2}{2\Delta}}.$$
\end{theorem}

\begin{remark} \label{remJanson}
For the special case where $H$ is a triangle, it is easy to show (by
fixing the intersection edge) that $\Delta\leq \mu_{\mathcal
S}+(\binom{m}{2}(n-m)^2+\binom{n-m}{2}n^2+m(n-m)n^2)p^5$. We make
use of this later.
\end{remark}

%\noindent The following is a trivial yet useful bound.
%\begin{lemma}\label{Che2}
%Let $X \sim \emph{\Bin}(n,p)$ and $k \in \mathbb{N}$. Then the
%following holds:
%\[\Pr(X\geq k) \leq \left(\frac{enp}{k}\right)^k.\]
%\end{lemma}
%\begin{proof}
%$\Pr(X \geq k) \leq \binom{n}{k}p^k \leq
%\left(\frac{enp}{k}\right)^k$.
%\end{proof}
Another tool to be used in our proofs is the following well known
lemma due to Borel and Cantelli.

\begin{lemma}[Borel-Cantelli Lemma]
Let $(A_i)_{i=1}^{\infty}$ be a sequence of events. Then
\begin{enumerate}[$(a)$]
\item If $\sum_{k}\Pr\left[A_k\right]<\infty$, then
$$\Pr\left[A_k \textrm{ for infinitely many }k\right]=0.$$
\item If $\sum_k \Pr\left[A_k\right]=\infty$ and in addition all
the $A_k$'s are independent, then
$$\Pr\left[A_k \textrm{ for infinitely many }k\right]=1.$$
\end{enumerate}
\end{lemma}

The following theorem due to Rinott \cite{Rinott} shows that, under some assumptions, the sum
of dependent random variables satisfies CLT, and measures the error
term based on the dependencies between the variables. Before stating
it explicitly, we need the following definition.

\begin{definition} \label{def:dependencygraph}
Let $(X_i)_{i\in I}$ be a collection of random variables. A graph
$D$ on a vertex set $I$ is called a \emph{dependency graph} for the
collection if for any pair of disjoint subsets $I_1,I_2\subseteq I$
for which there are no edges of $D$ between $I_1$ and $I_2$, the
random variables $(X_i)_{i\in I_1}$ and $(X_j)_{j\in I_2}$ are
independent.
\end{definition}

Now we state the result from
\cite{Rinott} which we are going to use.

\begin{theorem}[Theorem 2.2 in \cite{Rinott}]\label{thm:Rinott}
Let $(t_i)_{i=1}^n$ be a collection of random variables. Let
$X=\sum_{i=1}^n t_i$  and  $\mu:=\mathbb{E}(X)$ and $\sigma^2:=Var(X)>0$.
Let $D$ be a dependency graph for the collection and suppose that
$|t_i-\mathbb{E}(t_i)|\leq B$ a.s.\ for every $i$ and that
$\Delta(D)\leq C$. Then
$$\left|\Pr\left[\frac{X-\mu}{\sigma}\leq x\right]-\Phi(x)\right|\leq \frac
{BC}\sigma \left(\sqrt{\frac 1{2\pi}}+16
\left(\frac{n}{\sigma^2}\right)^{1/2}C^{1/2}B+10\left(\frac{n}{\sigma^2}\right)
CB^2\right).$$
\end{theorem}

\begin{remark}\label{rem1}
Note that whenever $\sigma^2=\Omega(nCB^2)$ the expression in the
right hand side of the inequality in Theorem \ref{thm:Rinott} is
$O\left(\frac{BC}{\sigma}\right)$. Assuming this, since
$\lim_{x\rightarrow \infty}\frac{\int_x^{\infty}e^{-t^2/2}dt}{\frac
1x e^{-x^2/2}}=1$, for large enough $x$ it follows by Theorem
\ref{thm:Rinott} that
$$\left|\Pr\left[\frac{X-\mu}{\sigma}\geq x\right]-\frac
{1}{x\sqrt{2\pi}} e^{-x^2/2}\right|=
O\left(\frac{BC}{\sigma}\right).$$
\end{remark}

The key tools in the proofs of Theorems  \ref{thm:main} and \ref{thm:main2} are the following concentration bounds, which may be of independent interest. We postpone their proofs to Section \ref{upper tail estimates}.
\begin{lemma}\label{conc0}Let $X_{n,m}$ be the number of perfect matchings in $B(n,m)$. Let $0<\delta<1/2$ be a constant. There is a constant $C$, depending on $\delta$, such that for any $\delta n^2\leq m\leq (1-\delta)n^2$, and $k=o(n^{1/3})$, we have
\[
\mathbb{E}[X_{n,m}^k]\leq C^k(\mathbb{E}[X_{n,m}])^k
\]
\end{lemma}Markov's bound implies that for $K\geq C$ one has:
\[
\Pr[X_{n,m}\geq K \mathbb{E}[X_{n,m}]]\leq (C/K)^k
\]by taking $\delta:=\min\{p/2,(1-p)/2\}$, $k=4\log n$, and $K=Ce$, we have the following corollary
\begin{corollary}\label{concentration0}Let $0<p<1$ be a constant. There is a constant $K$ (depending on $p$) such that for any $\frac{p}{2}n^2\leq m\leq \frac{1+p}{2}n^2$ one has
$$
\Pr[X_{n,m}\geq K\mathbb{E}[X_{n,m}]]\leq n^{-4}
$$
\end{corollary}
\noindent The concentration bounds for Hamilton cycles are as follows
\begin{lemma}\label{conc1} Let $X_{n,m}$ be the number of Hamilton cycles in $G(n,m)$. Let $0<\delta<1/2$ be a constant. There is a constant $C$, depending on $\delta$, such that for any $\delta {n \choose 2}\leq m\leq (1-\delta){n \choose 2}$, and $k\leq \frac{\log n}{8}$ we have:
\[
\mathbb{E}[X_{n,m}^k]\leq C^k(\mathbb{E}[X_{n,m}])^k
\]
\end{lemma}

Again, Markov's bound implies that for $K\geq C$ one has:
\[
\Pr[X_{n,m}]\geq K \mathbb{E}[X_{n,m}]]\leq (C/K)^k
\]by taking $\delta:=\min\{p/2,(1-p)/2\}$, $k=\frac{\log n}{8}$, and $K=Ce^{32}$, we have the following corollary
\begin{corollary}\label{concentration}Let $0<p<1$ be a constant. There is a constant $K$ (depending on $p$) such that for any $\frac{p}{2}{n\choose 2}\leq m\leq \frac{1+p}{2}{n\choose 2}$ one has
$$
\Pr[X_{n,m}\geq K\mathbb{E}[X_{n,m}]\leq n^{-4}
$$
\end{corollary}
The last lemma is an approximation to the lower factorial that we will use throughout.
\begin{lemma}\label{lower factorial}Let $t,\ell$ be integers such that $\ell=o(t^{2/3})$. Then,
\[
(t)_\ell=t^\ell\exp\left(-\frac{\ell(\ell-1)}{2t}+o(1)\right)
\]
\end{lemma}
In the proof of the upper-tail estimate for perfect matchings, we will need Bregman's theorem, which allows us to bound the number of perfect matchings given the degree sequence:

\begin{theorem}[Bregman-Minc inequality; \cite{Bregman}]\label{bregman}Let $G$ be a bipartite graph with two color classes $V=\{v_1,\ldots,v_n\}$ and $W=\{w_1,\ldots,w_n\}$. Denote by $M$ the number of perfect matchings and
$d_{v_i}$ the degree of $v_i$.  Then \[
M\leq \prod_{i=1}^n(d(v_i)!)^{1/d(v_i)}
\]
\end{theorem}

\section{Proof of Theorem \ref{main1}} \label{sec:proofs}

\begin{proof}
Let $H$ be a graph on $\ell$ vertices, where $\ell$ is a fixed
constant. For the sake of simplicity of notation, throughout the
whole proof we omit the up-script $G$ from the random variables. In
order to prove Theorem \ref{main1} we aim to show that for every
$\varepsilon>0$ we have both the upper bound
$$\Pr\left[\frac{X_n-\mu_n}{\sigma_n} \geq (1+\varepsilon)\sqrt{2\log\log n} \textrm{ for infinitely
many }n\right]=0,$$ and the lower bound
$$\Pr\left[\frac{X_n-\mu_n}{\sigma_n} \geq (1-\varepsilon)\sqrt{2\log\log n} \textrm{ for infinitely
many }n\right]=1.$$

Since throughout the proof we make use of Theorem \ref{thm:Rinott}
for estimating the upper tails of random variables of the form
$X_n-X_m$, it will be convenient to introduce some notation. For
every $n\geq m$ let $\mathcal S_{n,m}=\mathcal H_n\setminus \mathcal
H_m$, where $\mathcal S_{n,0}=\mathcal H_n$. Let us define a
dependency graph for $\mathcal S_{n,m}$ in the following manner. The
vertex set of $D_{n,m}$ is $\mathcal S_{n,m}$, and the edge set
consists of all pairs $s,t\in \mathcal S_{n,m}$ for which $|E(s)\cap
E(t)|\geq 1$ (that is, pairs of copies of $H$ which share at least
one edge). Note that it trivially follows from the way we labeld
$\mathcal H$ that $V(D_{n,m})=\mathcal S_{n,m}$ is the number of
copies of $H$ with at least one vertex taken from
$\{m+1,\ldots,n\}$. In addition, it is easy to see that
$$\Delta(D_{n,m})\leq
c'_H|E(H)|n^{|V(H)|-2}=\Theta(n^{\ell-2}),$$ where $c'_H$ is the
maximum number of automorphisms of $H$ preserving some edge. Now,
let us denote by $X_{n,m}:=X_{\mathcal S_{n,m}}$ and let
$\mu_{n,m}$ and $\sigma^2_{n,m}$ be its expectation and variance,
respectively. Trivially, we have $\mu_{n,m}=\mu_n-\mu_m$ and
$|\xi_t-\mathbb{E}(\xi_t)|\leq 1$ for every $t\in V(D_{n,m})$.
%Note
%that if we take $a^k\leq m\leq n\leq a^{k+1}$, where $a$ is at most
%(say) $2$, and consider $k$ to be sufficiently large, then
%$\sigma_{n,m}^2=\Theta(m(n-m)n^2)$ (this is verified in Section
%\ref{sec:calculations}).
Therefore, while applying Theorem \ref{thm:Rinott} for a large $x$
with $C=\Delta(D_{n,m})$ and $B=1$, using Remark \ref{rem1} we
obtain
\begin{equation} \label{eq:rinott}
\Pr\left[\frac{X_{n,m}-\mu_{n,m}}{\sigma_{n,m}}\geq x\right]=\frac 1{x\sqrt{2\pi}} e^{-x^2/2}+O\left(\frac{n^{\ell-2}}{\sigma_{n,m}}\right).
\end{equation}
Note that whenever we use \eqref{eq:rinott}, one should verify that the error
term is negligible compared to the first summand on the right hand
side. Most of the times it will be quite easy to check and therefore
we omit the calculations. For some relevant estimates on the variances that we use in the proof, the reader should consult Section \ref{sec:calculations}.

%Note that for $m=0$ it is quite easy to show that
%$\sigma_{n,m}=\Theta\left(n^{\ell-1}\right)$, and therefore, the
%error term on the right hand side of $(*)$ is simply $O(1/n)$ for
%this case.

%(and therefore, we also have
%$$\Pr\left[\frac{\xi_{n,m}-\mu_{n,m}}{\sigma_{n,m}}\geq x\right]=1-\frac 1{x\sqrt{2\pi}} e^{-x^2/2}+O\left(\frac{1}{\sqrt{m(n-m)}}\right).$$

%In a similar way, in the case $m=0$ and $n$ sufficiently large, we
%obtain
%$$\Pr\left[\frac{\xi_{n}-\mu_{n}}{\sigma_{n}}\geq x\right]=\frac 1{x\sqrt{2\pi}} e^{-x^2/2}+O\left(\frac{1}{n}\right).$$

Let us start with proving the upper bound.

\subsection{Upper bound}\label{sec:upper}

Let $\varepsilon>0$ be some positive constant and let
$x=(1+\varepsilon/4)\sqrt{2\log\log n}$. Note that for a fixed $n$, by distinguishing between the two cases $|h\cap h'|=2$ and $|h\cap h'|>2$, we obtain $$\sigma^2_n:=\sigma^2_{n,0}=\Theta\left(n^{\ell}+\sum_{h,h'\in \mathcal S_n}Cov(\xi_h\xi_h')\right)=\Theta\left(n^{\ell}+n^{2\ell-2}+n^{\ell}n^{\ell-3}\right).$$
Therefore, by \eqref{eq:rinott} we have
\begin{align*}\Pr\left[\frac{X_n-\mu_n}{\sigma_n}\geq
x\right]&=\frac 1{x\sqrt{2\pi}}
e^{-x^2/2}+O\left(\frac{1}{n}\right)\\
&=O\left((\log n)^{-(1+\varepsilon/4)^2}\right).
\end{align*}
Using this estimate for every (large enough) $n$ of the form $a^k$ (where $a>1$),
we obtain that

$$\sum_{k=1}^\infty \Pr\left[\frac{X_{a^k}-\mu_{a^k}}{\sigma_{a^k}}\geq
(1+\varepsilon/4)\sqrt{2\log\log {a^k}}\right]=\sum_{k=1}^\infty
O\left(k^{-(1+\varepsilon/4)^2}\right)<\infty ,$$

and therefore, it follows from the Borel-Cantelli Lemma that for
some $k_0\in \mathbb{N}$ we have

$$\Pr\left[\frac{X_{a^k}-\mu_{a^k}}{\sigma_{a^k}}\leq
(1+\varepsilon/4)\sqrt{2\log\log {a^k}} \textrm{ for all }k\geq
k_0\right]=1.$$

Note that if $a$ is not an integer then we always assume that $k$ is sufficiently large and we set $n=\lfloor a^k\rfloor$. As it does not affect any of our asymptotic calculations, we will omit the flooring signs. 

In order to complete the proof (of the upper bound), we need to
``close the gaps". That is, we need to show that there exists
$k_1\in \mathbb{N}$ such that

$$\Pr\left[\frac{X_{n}-\mu_{n}}{\sigma_{n}}\leq
(1+\varepsilon)\sqrt{2\log\log {n}}\textrm{ for every } n\geq
a^{k_1}\right]=1.$$

To this end we act in the following way. Fix $a>1$ which is close
enough to $1$ (to be determined later), and we show that
$\sum_{k}\delta_k<\infty$, where
$$\delta_k:=\Pr\left[X_{n,a^k}-\mu_{n,a^k}\geq
\frac{\varepsilon}2\sigma_n\sqrt{2\log\log {n}} \textrm{ for some
}a^k\leq n\leq a^{k+1}\right].$$

Therefore, using the Borel-Cantelli Lemma we conclude that there
exists $k_1$ for which

$$\Pr\left[X_{n,a^k}-\mu_{n,a^k}<
\frac{\varepsilon}2\sigma_n\sqrt{2\log\log {n}} \textrm{ for every }
k\geq k_1 \text{ and }a^k\leq n\leq a^{k+1}\right]=1.$$

Next, recall that
$$\Pr\left[X_{a^k}-\mu_{a^k}\leq
(1+\varepsilon/4)\sigma_{a^k}\sqrt{2\log\log {a^k}} \textrm{ for all
}k\geq k_0\right]=1,$$ and set $k_2:=\max\{k_0,k_1\}$.

All in all, we obtain that with probability 1, for every $k\geq k_2$
and for every $a^k\leq n\leq a^{k+1}$ we have
\begin{align*} X_n-\mu_n&=[(X_n-X_{a^k})-(\mu_n-\mu_{a^k})]+
(X_{a^k}-\mu_{a^k})\\&=(X_{n,a^k}-\mu_{n,a^k})+(X_{a^k}-\mu_{a^k})\\
&< \frac{\varepsilon}2\sigma_n\sqrt{2\log\log
{n}}+(1+\varepsilon/4)\sigma_{a^k}\sqrt{2\log\log n}\\
&< (1+\varepsilon)\sigma_n\sqrt{2\log\log n},
\end{align*}
as desired.

In order to complete our argument, we need to estimate $\delta_k$
and to show that indeed $\sum \delta_k<\infty$. This is done in the
following claim, which is a modification of Levy's inequality to our special case of dependent random variable.

\begin{claim} \label{claim1}
 $\delta_k$ is such that $\sum \delta_k<\infty$.
\end{claim}

\begin{proof} Fix $k\in \mathbb{N}$. For each $m\leq n$ and for each $\tau>0$, let $\mathcal
E_{n,m,\tau}$ denote the event $\{X_{n,m}-\mu_{n,m}\leq \tau\}$.
Let $n=a^{k+1}$, $\tau=\frac{\varepsilon}2\sigma_n\sqrt{2\log\log
n}$, and for every $a^k\leq j\leq a^{k+1}$ define

$$A_j:=\begin{cases}\left(\bigcap_{i=a^k}^{j-1}\mathcal E_{i,a^k,\tau}\right)\cap \neg\mathcal E_{j,a^k,\tau}   &\mbox{ for } j\geq a^k+1 \\
\neg\mathcal E_{a^k,a^k,\tau}& \mbox{ for } j=a^k.
\end{cases}$$

Note that we have $\tau=\sigma_n\cdot \omega(1)$ and that $\tau\leq \mu_{n,j}$, both will be used later in the proof. In order to see the latter, recall that 
$\mu_{n,j}=\Theta((n-j)n^{\ell-1})$ and $\sigma^2_n\leq n^{\ell}+O(n^{2\ell-2})$. Therefore, $\sigma_n=O(n^{\ell-1})= o(\mu_{n,j}/\log n)$ for all $j\leq n-\omega(\log n)$. 

Now, let $M_n:=\bigcup_{j=a^k}^n \neg\mathcal E_{j,a^k,\tau}$ and
note that $M_n=\bigcup_{j=a^k}^{n} A_j$ and that
$\delta_k=\Pr\left[M_{n}\right]$. We start with evaluating the
following probability:

\begin{align}\label{1} \Pr\left[M_{n}\cap \mathcal
E_{n,a^k,\tau/2}\right]=\sum_{j=a^k}^{n}\Pr\left[A_j\cap \mathcal
E_{n,a^k,\tau/2}\right].
\end{align}

Note that if $A_j\cap \mathcal E_{n,a^k,\tau/2}$, then in particular
we have
$$X_{j,a^k}-\mu_{j,a^k}> \tau \textrm{ and }
X_{n,a^k}-\mu_{n,a^k}\leq \tau/2.$$

Therefore, we conclude that
$$(X_{j,a^k}-\mu_{j,a^k})-(X_{n,a^k}-\mu_{n,a^k})>\tau/2,$$

which is equivalent to
\begin{align}\label{2}
X_{n,j}<\mu_{n,j}-\tau/2.
\end{align}

Moreover, a moment's thought reveals that for every $j$, the events
$\{X_{n,j}<\mu_{n,j}-\tau/2\}$ and $A_j$  are negatively
correlated, and therefore, one can upper bound \eqref{1} by
\begin{align}\label{3}
\sum_{j=a^k}^{n}\Pr\left[A_j \textrm{ and
}(X_{n,j}<\mu_{n,j}-\tau/2)\right]&\leq
\sum_{j=a^k}^{n}\Pr\left[A_j\right]\Pr\left[X_{n,j}<\mu_{n,j}-\tau/2\right].
\end{align}

Now, since clearly $\sigma_{n,j}\leq \sigma_n$, and since
$\tau=\sigma_n \cdot \omega(1)$, it follows by \eqref{eq:rinott}
that for every $a^k\leq j\leq n-\log^2n$ we have

\begin{align}\label{4}
\Pr\left[X_{n,j}<\mu_{n,j}-\tau/2\right]&=\Pr\left[\frac{X_{n,j}-\mu_{n,j}}{\sigma_{n,j}}<\frac{\tau}{2\sigma_{n,j}}\right]\nonumber \\
&=\exp\left(-\omega(1)\right)=o(1).
\end{align}

For larger values of $j$ we will simply observe that 
$$\Pr\left[X_{j,a^k}>\mu_{j,a_k}+\tau\right]=o(1),$$
as desired.

Combining \eqref{4} with \eqref{1} and \eqref{3}, we obtain

\begin{align} \label{5}
\Pr\left[M_{n}\cap \mathcal E_{n,a^k,\tau/2}\right]= \delta_k \cdot
o(1).
\end{align}

As a penultimate step, we need to estimate $\Pr\left[\neg\mathcal
E_{n,a^k,\tau/2}\right]$. In order to do so we first observe that
since we choose $a>1$ to be very close to $1$, it is easy to verify
that in this case we have
$\sigma^2_{n,j}=\Theta\left(j(n-j)n^{2\ell-4}\right)$ (while
$\sigma^2_n=\Theta\left(n^{2\ell-2}\right)$). Now, together with
\eqref{eq:rinott}, these estimates imply that for some small constant
$C:=C(\varepsilon)>0$ we have

\begin{align}\label{6}
\Pr\left[\neg\mathcal
E_{n,a^k,\tau/2}\right]&=\Pr\left[\frac{X_{n,a^k}-\mu_{n,a^k}}{\sigma_{n,a^k}}\geq
\tau/(2\sigma_{n,a^k})\right] \nonumber \\
&\leq \exp\left( -\frac{\tau^2}{8\sigma^2_{n,a^k}}\right)
\nonumber\\
&=\exp \left(-\frac{Ca^{4k}\log\log a^{k+1}}{a^{4k}(a-1)}\right),
\end{align}
and by choosing $a-1<C/2$, one can upper bound \eqref{6} with
$k^{-2}$ (for large $k$).

All in all, we obtain

\begin{align*}
\delta_k=\Pr\left[M_n\right]&=\Pr\left[M_n\cap \mathcal
E_{n,a^k,\tau/2}\right]+\Pr\left[M_n\cap \neg\mathcal
E_{n,a^k,\tau/2}\right]\\
&\leq \delta_k\cdot o(1)+k^{-2},
\end{align*}
and therefore, $\delta_k=O(k^{-2})$ and $\sum \delta_k<\infty$ as
desired. This completes the proof of the claim, and therefore the
proof of the upper bound as well.
\end{proof}

Before we proceed to the lower bound, let us make a few observations
which can be obtained in a similar way as the above proof. We make
use of those in the next subsection.

\begin{enumerate}[$(O1)$]
\item For every $\varepsilon>0$ we have $\Pr\left[X_n-\mu_n \leq
-(1+\varepsilon)\sigma_n\sqrt{2\log\log n} \textrm{ for infinitely
many }n\right]=0$.
\item For $k\in \mathbb{N}$, let $\zeta_k$ be the random variable
counting the number of copies of $H$ with vertices from both
$\{a^{k}+1,\ldots, a^{k+1}\}$ and $[a^k]$. Let us also denote by
$\widetilde {\mu}_k$ and $\widetilde{\sigma}_k^2$ its expectation
and variance, respectively. Then, for every $\varepsilon>0$ we have
$$\Pr\left[\zeta_k-\widetilde{\mu}_{k}\leq
-(1+\varepsilon)\widetilde{\sigma}_k\sqrt{2\log\log a^{k+1}}
\textrm{ for infinitely many }k\right]=0.$$
\end{enumerate}

\subsection{Lower bound}\label{sec:lower}

%For the convenience of the reader, we first prove the lower bound
%for the special (perhaps simplest) case where $H$ is a triangle. The
%advantage is, in this case, that it is fairly easy to do all the
%relevant calculations so we can provide the reader with a very
%detailed proof. In Section \ref{sec:generalcase} we explain how to
%adjust this proof to the general case, but we omit most of the
%(quite easy) calculations.

Let $\varepsilon>0$ be some fixed positive constant, we aim to show
that
$$\Pr\left[\frac{X_n-\mu_n}{\sigma_n}\geq
(1-\varepsilon)\sqrt{2\log\log n} \textrm{ for infinitely many
}n\right]=1.$$

To this end, we focus on integers $n_k$ of the form $a^k$, where
$a>1$ is a large enough constant to be determined later.

For a fixed $k\in \mathbb{N}$, let $\eta_k$ be the random variable
that counts the number of copies of $H$ which are fully contained in
$\{a^k+1,\ldots,a^{k+1}\}$. Note that the set $\{\eta_k :
k\in\mathbb{N}\}$ is clearly independent, and that the random
variables $\eta_k$ are distributed the same as $X_{a^{k+1}-a^k}$
(and therefore, $\sigma^2_{\eta_k}=\sigma^2_{a^{k+1}-a^k}$ for every
$k$). Therefore, one can easily check that for large $a$ and $k$ we
have
\begin{align}\label{lowerbound:1}(1-\varepsilon/4)\sigma_{\eta_k}\sqrt{2\log\log
(a^{k+1}-a^{k})}\geq
(1-\varepsilon/2)\sigma_{a^{k+1}}\sqrt{2\log\log a^{k+1}}
\end{align}
(this can be verified using the simple observation that
$\lim_{a,k\rightarrow \infty}\frac{\log\log
(a^{k}-a^{k-1})}{\log\log a^k}=1$ and the estimate \eqref{varXin}
given in Section \ref{sec:calculations}).

Now, letting $x=(1-\varepsilon/4)\sqrt{2\log\log (a^{k+1}-a^k)}$ it
follows by $(*)$ that for some $\gamma>0$ we have

$$\Pr\left[\frac{\eta_k-\mu_{\eta_k}}{\sigma_{\eta_k}}\geq
x\right]=\Omega\left(k^{-1+\gamma}\right),$$

and therefore,
$$\sum_k\Pr\left[\frac{\eta_k-\mu_{\eta_k}}{\sigma_{\eta_k}}\geq
x\right]=\infty.$$

Using the Borel-Cantelli Lemma it thus follows that

$$\Pr\left[\frac{\eta_k-\mu_{\eta_k}}{\sigma_{\eta_k}}\geq
x \textrm{ for infinitely many } k\right]=1.$$

%In order to complete the proof, note first that it follows directly
%from the proof of the upper bound and from Observations $(O1)$ and
%$(O2)$ that
%
%\begin{enumerate}[$(a)$]
%%\item $\xi_{a^{k-1}}-\mu_{a^{k-1}}\leq
%%(1+\varepsilon/2)\sigma_{a^{k-1}}\sqrt{2\log\log a^{k-1}}$, and
%\item $\zeta_{k-1}-\widetilde{\mu}_{k-1}\geq
%-(1+\varepsilon)\widetilde{\sigma}_{k-1}\sqrt{2\log\log a^{k}}$, and
%\item $\xi_{a^{k-1}}-\mu_{a^{k-1}}\geq-(1+\varepsilon)\sigma_{a^{k-1}}\sqrt{2\log\log
%a^{k-1}}$.
%\end{enumerate}

Now, let us choose $a>1$ to be a fixed large enough constant so
that for sufficiently large $k$ the following inequalities hold (the
existence of such $a$ for which all these inequalities hold follows
immediately from the relevant estimates in Section
\ref{sec:calculations}):

\begin{enumerate} [$(i)$]
\item $(1-\varepsilon/4)\sigma_{\eta_k}\sqrt{2\log\log (a^{k+1}-a^{k})}\geq (1-\varepsilon/2)\sigma_{a^{k+1}}\sqrt{2\log\log
a^{k+1}},$ and

\item $(1+\varepsilon)\sigma_{a^{k}}\sqrt{2\log\log a^{k}}\leq
(\varepsilon/4)\sigma_{a^{k+1}}\sqrt{2\log\log a^{k+1}},$ and

\item $(1+\varepsilon)\widetilde{\sigma}_{k}\sqrt{2\log\log a^{k+1}}\leq
(\varepsilon/4)\sigma_{a^{k+1}}\sqrt{2\log\log a^{k+1}}.$

\end{enumerate}

All in all, combining the above mentioned estimates and
$(i)$-$(iii)$ we conclude

\begin{align*}
X_{a^{k+1}}-\mu_{a^{k+1}}&=(\eta_k-\mu_{\eta_k})+(X_{a^{k}}-\mu_{a^{k}})+(\zeta_{k}-\widetilde{\mu}_{k})\\
&\geq
(1-\varepsilon/4)\sigma_{\eta_k}\sqrt{2\log\log(a^{k+1}-a^{k})}-(1+\varepsilon)\sigma_{a^{k}}\sqrt{2\log\log
a^{k}}-(1+\varepsilon)\widetilde{\sigma}_{k}\sqrt{2\log\log
a^{k+1}}\\
&\geq (1-\varepsilon/2)\sigma_{a^{k+1}}\sqrt{2\log\log
a^{k+1}}-(\varepsilon/2)\cdot\sigma_{a^{k+1}}\sqrt{2\log\log a^{k+1}}\\
&\geq (1-\varepsilon)\sigma_{a^{k+1}}\sqrt{2\log\log a^{k+1}},
\end{align*}
as desired. This completes the proof.
\end{proof}

\subsection{Relevant estimates for the variances appearing in the proof of theorem \ref{main1}}\label{sec:calculations} In this section we verify \eqref{lowerbound:1}, $(ii)$ and $(iii)$, by estimating the relevant variances.
Before doing so, recall that
\begin{align*}
Var(X_1+\ldots+X_n)=\sum_{i=1}^nVar(X_i)+\sum_{i\neq j}Cov(X_i,X_j),
\end{align*}
where $Cov(X,Y)=\mathbb{E}XY-\mathbb{E}X\mathbb{E}Y$. Moreover, note
that whenever $X$ and $Y$ are independent, then $Cov(X,Y)=0$.
Therefore, given a subset $\mathcal S\subseteq \mathcal H$, it
follows that
\begin{align}\label{sumVar}
Var(X_{\mathcal S})&=\sum_{t\in \mathcal S}Var(\xi_t)+\sum_{t\neq
s \textrm{ and } E(t)\cap E(s)\neq \emptyset}Cov(\xi_t,\xi_s)
\nonumber \\
&=\sum_{t\in \mathcal
S}Var(\xi_t)+\sum_{i=1}^{\ell-1}\sum_{|E(t)\cap
E(s)|=i}Cov(\xi_t,\xi_s).
\end{align}

In addition, recall that each of the $\xi_t$'s is an indicator
random variable for an appearance of a certain copy of $H$ (where
$|V(H)|=\ell$ and $|E(H)|=m$), and therefore we have
\begin{align*}
\mathbb{E}\xi_t=p^m \textrm{ and } Var(\xi_t)=p^m(1-p^m)=p^m-p^{2m}.
\end{align*}

Next, recall that $p$ and $\ell:=|V(H)|$ are fixed constants and
that we always assume $a$ and $k$ to be large enough. In particular,
it easy to see that the (asymptotically) largest element in the right
hand side of \eqref{sumVar} is the case $i=1$.

Now we can give some easy estimates.

{\bf Estimating $\sigma^2_n:=Var(X_n)$:} Recall that $X_n$ is a
sum of indicator random variables for all the (labeled) copies of
$H$ in $K_n$. Therefore, there exists a constant $C$ (which depend
of the number of automorphisms which preserve some edge) such that
the number of pairs $(s,t)$ of copies of $H$ which intersect in
exactly one edge is roughly $(1+o(1))Cn^{2\ell-2}$. Therefore,
running over all possible intersection edges we obtain that
\begin{align}\label{varXin}
\sigma^2_n&=(1+o(1))Cn^{2\ell}(p^{2m-1}-p^{2m}).
\end{align}

Now, note that since
\begin{align*}
  \sigma^2_{a^{k+1}-a^k}&=(1+o(1))C(a^{k+1}-a^k)^{2\ell}(p^{2m-1}-p^{2m})\\
  &=(1+o(1))C(a^k(a-1))^{2\ell}(p^{2m-1}-p^{2m}),
\end{align*}
by taking $a$ to be sufficiently large we obtain that
$$\sigma^2_{a^{k+1}-a^k}=(1+o(1))Ca^{2\ell(k+1)}(p^{2m-1}-p^m)$$
which is of the same order of magnitude as $\sigma^2_{a^{k+1}}$.
This verifies  \eqref{lowerbound:1}.

In order to verify $(ii)$ all we need is to note that the quantity
$\sigma^2_{a^{k+1}}/\sigma^2_{a^k}$ is a function that tends to
infinity whenever $a$ does.

Finally, in order to verify $(iii)$ let us first estimate
$\widetilde{\sigma}^2_k$.

{\bf Estimating $\widetilde{\sigma}^2_{k}:=Var(\zeta_k)$:} Let $k\in
\mathbb{N}$ and $a>0$. Recall that $\zeta_k$ counts the number of
copies of $H$ with vertices from both $\{a^k+1,\ldots,a^{k+1}\}$ and
$[a^k]$. In this case, assuming $a$ goes to infinity, it is easy to
see that the largest summand in \ref{sumVar} is obtained whenever
the intersection edge is between $[a^k]$ and
$\{a^k+1,\ldots,a^{k+1}\}$. Therefore, for some constant $C'$ (which
does not depend on $a$) we obtain

\begin{align}\label{varZeta}
\widetilde{\sigma}^2_k&=(1+o(1))C'a^k(a^{k+1}-a^k)\left(a^{k+1}\right)^{2\ell-2}(p^{2m-1}-p^{2m})\nonumber\\
&=(1+o(1))C'\frac{1}{a}\left(a^{k+1}\right)^{2\ell}(p^{2m-1}-p^{2m}).
\end{align}

Note that by \eqref{varZeta} and \eqref{varXin} it follows that
$\widetilde{\sigma}^2_k=\Theta\left(\frac 1a
\sigma^2_{a^{k+1}}\right)$, and therefore, by taking $a$ to be
sufficiently large, $(iii)$ trivially holds.

\section{Proof of Theorem \ref{thm:main}}\label{perfect matchings}
Throughout the next section we are going to let $X_{n,m}$ be the number of perfect matchings in $B(n,m)$, $X_n$ the number of perfect matchings  in $B(n,p)$ and $Y_n:=\log X_n$. We aim to prove:
\[
\Pr  \left[\limsup_{n\rightarrow \infty}\frac{ Y_n- \log (n! p^n ) + \frac{1-p}{2p}  }{ \sqrt{2\log \log n}  \sqrt { \frac{1-p}{p}} }=1\right]=1
\]It will be enough to show that for
$\varepsilon>0$ we have both the upper bound
\begin{align}\Pr\left[\frac{ Y_n- \log (n! p^n ) + \frac{1-p}{2p}  }{  \sqrt { \frac{1-p}{p}} } \geq (1+\varepsilon)\sqrt{2\log\log n^2} \textrm{ for infinitely
many }n\right]=0,\nonumber
\end{align}
and the lower bound
\begin{align}\Pr\left[\frac{ Y_n- \log (n! p^n ) + \frac{1-p}{2p}  }{  \sqrt { \frac{1-p}{p}} } \geq (1-\varepsilon)\sqrt{2\log\log n^2} \textrm{ for infinitely
many }n\right]=1.\nonumber
\end{align}Note that in the equations above we have $\log\log{n^2}$, but those can be replaced by $\log\log n$ since the two quantities are asymptotically equal.
\subsection{Upper Bound}\label{upper bound section pm}

We need to prove that for  any fixed  $\varepsilon>0$

\begin{equation} \label{upper pm}
\Pr  \left[\frac{\log X_n-\log  (n! p^n)+\frac{1-p}{2p} }{ \sqrt { \frac{1-p}{p} }}\geq (1+\varepsilon)\sqrt{2 \log \log n^2}\text{ for infinite many }n\right]=0. \end{equation}

\noindent By Corollary \ref{concentration0}, there is a constant $K$ such that  for all $ \frac{p}{2} n^2 \le  m  \le \frac{1+p}{2} n^2 $

$$   X_{n,m}  \le K \mathbb{E}[X_{n,m}]  $$ with probability at least $1 -n^{-4} $. Taking $\log$,  we conclude  that with the same probability

\begin{equation} \label{upper pm1}
\log X_{n,m} \le \log \mathbb{E}[X_{n,m}] + \log K. \end{equation}

\noindent We use the following approximation of the expected value,
\[
\mathbb{E}[X_{n,m}]=n!p_m^n\exp\left(-\frac{1-p_m}{2p_m}+O(1/n)\right)
\]
\noindent (where  $p_m := \frac{m}{n^2}$). The calculation for which can be found in the Appendix. This yields,
$$\log \mathbb{E}[ X_{n,m}]  =\log (n! p_m^n) -  \frac{1-p_m}{2p_m } + o(1), $$The  RHS can be written as

$$ \log(n!)+n\log\frac{m}{n^2}-\frac{n^2}{2}\left(\frac{1}{m}-\frac{1}{n^2}\right)+ o(1). $$

\noindent Let $E_n$ be the random variable that counts the number of edges in $B(n,p)$. By  conditioning on $E_n=m$ and using the union bound
(over the range $\frac{p}{2}n^2 \le m \le \frac{1+p}{2} n^2$), we can conclude that  with probability at least  $1 -n^{-2} $

$$ \mathbb{I}_{\mathcal E}  \log X_n \le  \mathbb{I}_{\mathcal E}  \Big( \log(n!)+n\log\frac{E_n}{n^2}-\frac{n^2}{2}\left(\frac{1}{E_n}-\frac{1}{n^2} \right) + \log K +  o(1) \Big), $$

\noindent where $X_n$ denotes  the number of perfect matchings in $B(n,p)$,  and $\mathbb{I}_{\mathcal E} $ is the indicator of the event $\mathcal E$ that
$B(n,p)$ has at least $\frac{p}{2} n^2$ and at most $\frac{1+ p}{2} n^2 $ edges. By Chernoff's bound, $\mathbb{I}_{\mathcal E} = 1$ with probability at least $1 -n^{-2} $. By the union bound

\begin{align} \label{upper0}  \log X_n \le \Big( \log(n!)+n\log\frac{E_n}{n^2}-\frac{n^2}{2}\left(\frac{1}{E_n}-\frac{1}{n^2}\right) + O(1) \Big),
\end{align}  with probability at least $1 - 2n^{-2} $. Then,
\begin{align}
\log \frac{E_n}{n^2}&=\log \left(\frac{\sqrt{Var(E_n)}E_n^*}{n^2}+\frac{\mathbb{E}[ E_n]}{n^2}\right)\nonumber\\
&=\log\left(\left(\frac{p(1-p)}{n^2}\right)^{1/2}E_n^*+p\right)\nonumber\\
&=\log\left(p\left(\frac{1-p}{p}\right)^{1/2}\frac{E_n^*}{n}+p\right)\nonumber\\
&=\log p+\log \left(1+\left(\frac{1-p}{p}\right)^{1/2}\frac{E_n^*}{n}\right)\nonumber\\
&=\log p+\left(\frac{1-p}{p}\right)^{1/2}\frac{E_n^*}{n}+O(1/n^2).  \nonumber
\end{align}

\noindent Plugging the last estimate into  \eqref{upper0} we obtain,  with the same probability
\
\[
\log X_n \le \log (n! p^n) +\left(\frac{1-p}{p}\right)^{1/2}E_n^*-\frac{n^2}{2}\left(\frac{1}{E_n}-\frac{1}{n^2p}-\frac{p-1}{n^2p}\right) + O(1). \nonumber
\]
\noindent Note that with probability at least $1-n^{-2}$ we have $E_n=n^2p+O(n\log^2 n)$, in which case $\frac{n^2}{2}\left(\frac{1}{E_n}-\frac{1}{n^2p}\right)$ becomes $o(1)$. Thus, with probability at least $1-3n^{-2}$ we obtain
\begin{align}\label{comparison}
\frac{\log X_n-\log  (n! p^n)+\frac{1-p}{2p} }{ \sqrt { \frac{1-p}{p} }}\leq E_n^*+ O(1).
\end{align}

Since $\sum _n n^{-2} < \infty$, we have, by the Borell-Cantelli lemma that the event in \eqref{comparison} holds with probability 1 for all sufficiently large $n$.
On the other hand, by the Kolmogorov-Khinchin theorem,  $E_n^*$ satisfies LIL and thus
\[
E_n^*\leq (1+\varepsilon/2)\sqrt{2\log\log n^2}
\]happens with probability $1$ for all sufficiently large  $n$. For all sufficiently large $n$, $ (\varepsilon/2)\sqrt{2\log\log n^2 }$ is larger than the error term
$O(1)$, and we have
\[
\frac{\log X_n-\log (n! p^n)+\frac{1-p}{2p} }{\sqrt {\frac{1-p}{p}}}   \leq (1+\varepsilon)\sqrt{2\log\log n^2 },
\] proving equation (\ref{upper pm}).

\subsection{Proof of the Lower bound}

For the lower bound we need to show that there exists a sequence $n_k, k=1, 2 \dots$ of indices such that with probability 1,

\[
\frac{\log X_{n_k} -\log (n_k ! p^{n_k} )+\frac{1-p}{2p} }{\sqrt {\frac{1-p}{p}}}   \ge (1- \varepsilon)\sqrt{2\log\log n_k^2 },
\]

\noindent holds for infinitely many $k$.

Let $C>0$ be a constant. By the proof of \cite[Theorem 15]{Janson}, we know
\begin{align}\label{lower bound eqtn}
E_n^*-\frac{\log X_{n} -\log (n ! p^{n} )+\frac{1-p}{2p} }{\sqrt {\frac{1-p}{p}}}>C
\end{align}happens with probability $O(1/n)$, and $E_n^*$ is as in the last section. From  the standard proof of LIL for the sum of iid random variables \cite{Khinchin, Kolmogorov}, we see that there is a sequence $\{n_k\}:=\{c^k\}$ (where $c$ is an integer larger than 1) for which we have:
\[
E_{n_k}^*\geq (1-\varepsilon/2)\sqrt{2\log \log n_k^2}
\]
happens infinitely often with probability one. Restricting ourselves to this subsequence and denoting by $A_k$ the event that $\eqref{lower bound eqtn}$ holds for $n_k$, we have
\[
\Pr [A_k]=O(1/c^k)
\]so in particular we have
\[
\sum_{k} \Pr[A_k] <\infty
\] By  Borel-Cantelli lemma, we have that with probability equal to $1$, for all large $k$:
\[
E_{n_k}^*-C\leq \frac{\log X_{n_k} -\log (n_k ! p^{n} )+\frac{1-p}{2p} }{\sqrt {\frac{1-p}{p}}}
\]
\noindent Let $k$ be large enough so that $C<(\varepsilon/2)\sqrt{2\log \log n_k^2}$. Then, with probability equal to $1$ we have that for infinite many $k$:
\[
(1-\varepsilon)\sqrt{2\log \log n_k}\leq \frac{\log X_{n_k} -\log (n_k ! p^{n_k} )+\frac{1-p}{2p} }{\sqrt {\frac{1-p}{p}}}
\]just as desired.
\section{Proof of Theorem \ref{thm:main2}}\label{sec: main2}

Throughout the next section we are going to let $X_{n,m}$ be the number of Hamilton cycles in $G(n,m)$, $X_n$ the number of Hamilton cycles in $G(n,p)$ and $Y_n:=\log X_n$. The structure of the proof is identical to the one done for theorem \eqref{thm:main}, so we omit some of the calculations. We aim to prove:
\[
\Pr\left[\limsup_{n\rightarrow\infty}\frac{Y_n-\log \mathbb{E}[X_n]+\frac{1-p}{p}}{\sqrt{\frac{2(1-p)}{p}}\sqrt{2\log \log {n}}}=1\right]=1
\]It will be enough to show that
$\varepsilon>0$ we have both the upper bound
\begin{align}\Pr\left[\frac{\log X_n-\log \mathbb{E}[X_n]+\frac{1-p}{p}}{\sqrt{\frac{2(1-p)}{p}}} \geq (1+\varepsilon)\sqrt{2\log\log {n\choose 2}} \textrm{ for infinitely
many }n\right]=0,\nonumber
\end{align}
and the lower bound
\begin{align}\Pr\left[\frac{\log X_n-\log \mathbb{E}[X_n]+\frac{1-p}{p}}{{\sqrt{\frac{2(1-p)}{p}}}} \geq (1-\varepsilon)\sqrt{2\log\log {n\choose 2}} \textrm{ for infinitely
many }n\right]=1.\nonumber
\end{align}Note that in the equations above we have $\log\log{n\choose 2}$, but those can be replaced by $\log\log n$ since the two quantities are asymptotically equal.

\subsection{Proof of upper bound}\label{sec:ub hamilton}

Let $\varepsilon>0$, and let $N:=(n-1)!/2$ be the number of Hamilton cycles in the complete graph $K_n$. With this notation one has,
\begin{align}\label{ey}
  \mathbb{E}[X_{n,m}]=Np_m^n\exp\left(-\frac{n^2}{2m}(1-p_m)+o(1)\right).
\end{align}where in this section $p_m:=m/{n\choose 2}$. For a proof of \eqref{ey}, the reader can check the Appendix.
By using corollary \ref{concentration}, we have
\[
X_{n,m}\leq K \mathbb{E}[X_{n,m}]
\]
with probability at least $1-n^{-4}$.

Applying the log function and using estimate (\ref{ey}) we obtain
\begin{align}\label{logy}
\log X_{n,m}\leq \log K+\log N+n\log \frac{m}{{n\choose 2}}-\frac{n^2}{2}\left(\frac{1}{m}-\frac{1}{{n\choose 2}}\right)+o(1).
\end{align}Let $E_n$ be the random variable which counts the number of edges in $G\sim G(n,p)$, by conditioning on $E_n=m$ and using union bound (over the range $\frac{p}{2}{n\choose 2}\leq m\leq \frac{1+p}{2}{n\choose 2}$), with probability at least $1-n^{-2}$ we have
\begin{align}
\mathbb{I}_{\mathcal{E}}\log X_n\leq \mathbb{I}_{\mathcal{E}}\left(\log K+\log N+n\log \frac{E_n}{{n\choose 2}}-\frac{n^2}{2}\left(\frac{1}{E_n}-\frac{1}{{n\choose 2}}\right)+o(1)\right)
\end{align}Where now we use $X_n$ (number of Hamilton cycles in $ G(n,p)$) and $\mathbb{I}_{\mathcal{E}}$ is the indicator random variable that the number of edges in $G(n,p)$ is in the range $[\frac{p}{2}{n\choose 2},\frac{1+p}{2}{n\choose 2}]$. By Chernoff's bound, $\mathbb{I}_{\mathcal{E}}=1$ with probability at least $1-n^{-2}$. Hence, by the union bound we have
\begin{align}\label{logx}
\log X_n\leq \log K+\log N+n\log \frac{E_n}{{n\choose 2}}-\frac{n^2}{2}\left(\frac{1}{E_n}-\frac{1}{{n\choose 2}}\right)+o(1)
\end{align}with probability at least $1-2n^{-2}$. By a similar calculation to the one done in section \eqref{upper bound section pm}, we get
\begin{align}
\log\frac{E_n}{{n\choose 2}}&=\log p+\left(\frac{1-p}{{n\choose 2}p}\right)^{1/2}E_n^*+O(1/n^2).\nonumber
\end{align}
Plugging it into \eqref{logx}, we obtain that

\begin{align}
\log X_n&\leq \log K+\log N+n\left(\log p+\left(\frac{1-p}{{n\choose 2}p}\right)^{1/2}E_n^*\right)-\frac{n^2}{2}\left(\frac{1}{E_n}-\frac{1}{{n\choose 2}}\right)+o(1)\nonumber \\
&=\log \mathbb{E}[X_n]+\left(\frac{2(1-p)}{p}\right)^{1/2}E_n^*-\frac{n^2}{2}\left(\frac{1}{E_n}-\frac{1}{{n\choose 2}p}-\frac{p-1}{{n\choose 2}p}\right)+O(1).\nonumber
\end{align}

Note that since with probability $1-o(1/n^2)$ we have that (say) $E_n=m+\Theta(n\log^2 n)$, it follows that  $\frac{n^2}{2}\left(\frac{1}{E_n}-\frac{1}{{n\choose 2}p}\right)=o(1)$ without affecting the error probability.

%\[
%\log X\leq \log t+\log %\mathbb{E}[X]+\left(\frac{2(1-p)}{p}\right)^%{1/2}M^*+\frac{p-1}{p}+O(1/n)
%\]
All in all, with probability $1-O(1/n^2)$ we have
\begin{align}\label{ineq}
\frac{\log X_n-\log\mathbb{E}[X_n] +\frac{1-p}{p}}{\sqrt{\frac{2(1-p)}{p}}}\leq E_n^*+O(1).
\end{align}
By the Borel-Cantelli lemma, we see that for large $n$, with probability one, equation $\eqref{ineq}$ holds. Since $E_n^*$ satisfies LIL, we can upper bound the RHS of of $\eqref{ineq}$ by $(1+\varepsilon)\sqrt{2\log \log {n\choose 2}}$ for large $n$ with probability one. All in all,
\[
\frac{\log X_n-\log\mathbb{E}[X_n] +\frac{1-p}{p}}{\sqrt{\frac{2(1-p)}{p}}}\leq(1+\varepsilon)\sqrt{2\log \log {n\choose 2}}
\]holds for all large $n$ with probability one, which proves the upper bound.
\subsection{Proof of lower bound}\label{sec:lb hamilton}
Recall that in order to prove the lower bound one needs to show that for every $\varepsilon>0$ we have
\begin{align}\label{lb}
    \Pr\left[\frac{\log X_n-\log \mathbb{E}[X_n]+\frac{1-p}{p}}{{\sqrt{\frac{2(1-p)}{p}}}}\geq (1-\varepsilon)\sqrt{2\log \log {n\choose 2}} \text{ for infinite many }n\right]=1
\end{align}By the proof of [\cite{Janson},Theorem 1] we have that for any fixed constant $C>0$:
\[
\Pr\left[E_n^*-\frac{\log X_{n} -\log \mathbb{E}[X_n]+\frac{1-p}{p} }{\sqrt{\frac{2(1-p)}{p}}}>C\right]=O(1/n)
\]
By repeating the idea of the lower bound on theorem \eqref{main1}, we obtain
\[
\frac{\log X_n-\log \mathbb{E}[X_n]+\frac{1-p}{p}}{\sqrt{\frac{2(1-p)}{p}}}\geq (1-\varepsilon)\sqrt{2\log \log {n\choose 2}}\]holds for infinite many $n$ with probability $1$, which proves \eqref{lb}.

\section{Proof of Theorem \ref{main:CLT}}\label{sec:clt hypergraphs}

\begin{proof}In this section we will be working with loose Hamilton cycles in random hypergraphs $H^k(n,p)$. Note that we require that $m:=n/(k-1)$ is an integer (which shall denote the number of edges of a Hamilton cycle). Thus, we will assume the divisibility condition $k-1\mid n$ throughout the rest of the section. Let $\mathcal H$ be the set of all Hamilton cycles in the complete $k$-uniform hypergraph on $n$ vertices. Then,
\begin{align}\label{size}
|\mathcal H|=\frac{n!}{2m((k-2)!)^m}
\end{align}Indeed, there are $n!$ ways to label the vertices consecutively (and the edges are determined trivially, including the one edge which goes back to the beginning of the labeling). In each of the $m$ edges, for the ``non-overlapping" vertices (there are $k-2$ such vertices), the order is not important. Therefore, one should divide by $(k-2)!^m$. Finally, note that each Hamilton cycle can be obtained in $2m$ ways ($m$ ``overlapping vertices" to be placed as vertex number $1$, and two isomorphic ways to label the vertices consecutively).

Now we are ready to prove Theorem \ref{main:CLT}. Let $E_n$ denote the number of edges of $H^k(n,p)$, and $X_n(k):=X_n$ be the number of Hamilton cycles of $H^k(n,p)$. The idea of the proof is to compare $X_n$ to $E_n$. Specifically, we want to show that
\begin{align}
    \label{distance}
\mathbb{E}[|X_n^*-E_n^*|^2]
\end{align}goes to zero. Since clearly $E_n^*$ converges to $N(0,1)$, the theorem will follow.

To this end we will show that $X_n^*$ and $E_n^*$ are almost perfectly linearly correlated. Meaning that $Cov(X_n^*,E_n^*)\rightarrow 1$. Recall that
\begin{align}
    \label{covariance}
    Cov(X^*,E^*)=\frac{\mathbb{E}[X_nE_n]-\mathbb{E}[X_n]\mathbb{E}[E_n]}{\sqrt{Var(X_n)Var(E_n)}}.
\end{align}
Let $X_H$ be the event ``$H$ appears in $H^k(n,p)$". Hence,
\[
\mathbb{E}[X_H]=p^m
\]
Let $N:=\frac{n!}{2m((k-2)!)^m}$ (that is, $N=|\mathcal H|$), and by linearity of expectation, we have:
\[
\mathbb{E}[X_n]=Np^{m}
\] Also, since $E_n\sim Bi({n\choose k},p)$ we have $Var(E_n)={n\choose k} p(1-p)$ and $\mathbb{E}[E_n]={n\choose k} p$. We compute the missing quantities. Denote by $\mathcal E$ the set of edges in the complete $k$-uniform hypergraph, and denote by $E_e$ the event ``The edge $e$ appears in $H^k(n,p)$". Then,
\[
\mathbb{E}[X_nE_n]=\sum_{H\in\mathcal H,e\in\mathcal E} \mathbb{E}[X_H\cdot E_e]
\]By symmetry, by fixing one Hamilton cycle $H\in \mathcal H$, we have:
\[
\mathbb{E}[X_nE_n]=N\left(\sum_{e\in\mathcal E} p^{|H\cup e|}\right)=N\left(\left( {n\choose k} - m\right) p^{m+1}+mp^{m}\right)
\]Hence, $\mathbb{E}[X_nE_n]=\mathbb{E}[X_n](\mathbb{E}[E_n]+m(1-p))$, and we get $Cov(X_n,E_n)=\mathbb{E}[X_n](m(1-p))$. Lastly, we compute the variance of $X_n$.
\[
\mathbb{E}[X_n^2]=\sum_{H_1,H_2\in\mathcal H} p^{|H_1\cup H_2|}
\]Again, by fixing an arbitrary Hamilton cycle $H$, we get
\[
\mathbb{E}[X_n^2]=N\left(\sum_{H_1}p^{|H\cup H_1|}\right)
\]Let $N(a)$ be the number of Hamilton cycles that intersect $H$ in exactly $a$ edges. With this notation,
\[
\mathbb{E}[X_n^2]=N\left(\sum_{a= 0}^{m} N(a) p^{2m-a}\right)
\]Let $\alpha_a:=N(a)/N$. Then,
\[
\mathbb{E}[X_n^2]=N^2p^{2m}\left(\sum_{a=0}^{m} \alpha_a p^{-a}\right)
\]Hence,
\[
Var(X_n)=(\mathbb{E}[X_n])^2\left(-1+\left(\sum_{a= 0}^{m} \alpha_a p^{-a}\right)\right):=(\mathbb{E}[X_n])^2f(n)
\]Plugging back into \eqref{covariance}:
\begin{align}\label{cov}
Cov(X_n^*,E_n^*)=\frac{Cov(X_n,E_n)}{\sqrt{Var(X_n)Var(E_n)}}=\frac{\mathbb{E}[X_n](m)(1-p)}{\sqrt{(\mathbb{E}[X_n])^2f(n){n\choose k}p(1-p)}}=\frac{m(1-p)}{\sqrt{{n\choose k}p(1-p)f(n)}}
\end{align}Writing out $f(n)$:
\[
f(n)=(\alpha_0-1)+\frac{\alpha_1}{p}+\frac{\alpha_2}{p^2}+...=\alpha_1\left(\frac{1}{p}-1\right)+\alpha_2\left(\frac{1}{p^2}-1\right)+\ldots+\alpha_{m}\left(\frac{1}{p^{m}}-1\right)
\]Hence,
\[
f(n)\leq \alpha_1\left(\frac{1}{p}-1\right)+\sum_{t=2}^m\frac{\alpha_t}{p^t}
\]We are going to show that the sum is negligible compared to the first summand. First of all, note that $\alpha_1\leq m^2/{n\choose k}$ by a simple union bound. In general, to bound $\alpha_t$, we pick the $t$ edges from $H$ we are going to intersect. There are ${m\choose t}$ ways to do so. Next, collapse each one of those edges into a single vertex. Thus, we now have $n-t(k-1)$ vertices. Note that the number of vertices is still divisible by $k-1$, as it should be the case. Next, we form a Hamilton cycle on these vertices. There are
\[
\frac{(n-t(k-1))!}{2(m-t)((k-2)!)^{m-t}}
\]ways to do so. In order to see this, just note that we replace $n$ by $n-t(k-1)$ and $m$ by $m-t$ in equation (\ref{size}). Lastly, once the Hamilton cycle has been formed, we can uncollapse each one of the $t$ edges, so we obtain an extra factor of $(k!)^t$. Hence,
\begin{align}
\alpha_t&\leq \frac{1}{N}\cdot{m\choose t}\frac{(n-t(k-1))!(k!)^t}{2(m-t)((k-2)!)^{m-t}}\nonumber\\
&=\frac{(m)_t(k!)^tm((k-2)!)^t}{t!(m-t)(n)_{(k-1)t}}\nonumber\\
&=\frac{m(m)_tC^t}{(m-t)t!(n)_{(k-1)t}}
\end{align}for a constant $C$ depending on $k$. Plugging back on $f(n)$ we get:
\[
f(n)\leq \frac{m^2}{{n\choose k}}\left(\frac{1-p}{p}\right)+\sum_{t=2}^m\frac{m(m)_tC^t}{p^t(m-t)t!(n)_{(k-1)t}}
\]To handle the summation, we are going to split it into two sums:
\[
\sum_{t=2}^{\log n}\frac{m(m)_tC^t}{p^t(m-t)t!(n)_{(k-1)t}}+\sum_{t>\log n}^m\frac{m(m)_tC^t}{p^t(m-t)t!(n)_{(k-1)t}}:=S_1+S_2
\]Note that in the range $2\leq t\leq \log n$, we have by lemma \ref{lower factorial}:
\begin{itemize}
\item $(m)_t=m^t(1+o(1))$,
\item $(n)_{(k-1)t}=(1+o(1))n^{(k-1)t}$, and
\item $m/(m-t)\leq 2$.
\end{itemize}Hence,
\begin{align}
S_1&\leq(1+o(1))\sum_{t=2}^{\log n}\frac{2m^tC^t}{n^{(k-1)t}p^tt!}=O\left(\frac{m^2}{n^{2(k-1)}}\right)=O\left(\frac{1}{n^{2(k-2)}}\right)
\end{align}For $S_2$, we can upper bound $m/(m-t)\leq n$, and $(m)_t/(n)_{(k-1)t}\leq 1$ to obtain:
\[
S_2\leq \sum_{t>\log n}^m\frac{nC^t}{p^tt!}\leq \frac{n^2C^{\log n}}{p^{\log n}(\log n)!}=o\left(\frac{1}{n^{2(k-2)}}\right)
\]using this in the definition of $f$ we obtain:
\begin{align}
f(n)&\leq \frac{m^2}{{n\choose k}}\left(\frac{1-p}{p}\right)+O\left(\frac{1}{n^{2(k-2)}}\right)\nonumber\\
&=\frac{m^2}{{n\choose k}}\left(\frac{1-p}{p}\right)\left(1+O\left(n^{2-k}\right)\right)
\end{align}Thus,
\[
\frac{1}{\sqrt{1+O\left(n^{2-k}\right)}}\leq Cov(X_n^*,E_n^*)\leq 1
\]where the second inequality is just from Cauchy Schwarz. Then we have that the lower bound is:
\[
\frac{1}{\sqrt{1+O\left(n^{2-k}\right)}}
\]which we can re-write using a Taylor expansion as:
\begin{align}
    \label{convergence}
    1-O(n^{2-k})
\end{align}Hence, expanding \eqref{distance} and using \eqref{convergence} we have:
\[
\mathbb{E}[|X_n^*-E_n^*|^2]=\mathbb{E}[(X_n^*)^2]+\mathbb{E}[(E_n^*)^2]-2Cov(X_n^*,E_n^*)=2-2(1-O(n^{2-k})=O(n^{2-k})
\]
Hence, when $k\geq 3$, we have that the above tends to zero. This completes the proof of Theorem \ref{main:CLT}.
\end{proof}

\section{Proof of Theorem \ref{main:ILL}}\label{sec: ill hypergraph}

\begin{proof} Now we are going to use Theorem \ref{main:CLT} to derive LIL for $X_n^*$. First we note that since $E$ is the summation of ${n\choose k}$ i.i.d. random variables, then we have that $E_n^*$ obeys the LIL. That is,
\[
E_n^*\leq (1+\varepsilon/2)\sqrt{2\log\log {n}}
\]with probability $1$ for large enough $n$ and with probability 1 we also have
\[
E_n^*\geq (1-\varepsilon/2)\sqrt{2\log \log n}
\] infinitely often. Note that we write $\log\log n$ instead of $\log \log {n\choose k}$, which holds because they are asymptotically equal (as $k$ is fixed). Furthermore,
\[
\Pr\left(|X_n^*-E_n^*|\geq t)\leq \Pr(|X_n^*-E_n^*|^2\geq t^2\right)\leq \frac{\mathbb{E}[(X_n^*-E_n^*)^2]}{t^2}=O\left(\frac{1}{t^2n^{k-2}}\right)
\]let $t=(\varepsilon/2)\sqrt{2\log\log n}$. We obtain:
\begin{align}\label{errorterm}
\Pr\left(|X_n^*-E_n^*|\geq (\varepsilon/2)\sqrt{2\log\log n}\right)\leq O\left(\frac{1}{n^{k-2}\log\log n}\right)
\end{align}if $k\geq 4$, then we have:
\[
\sum_{n} \Pr(|X_n^*-E_n^*|\geq (\varepsilon/2)\sqrt{2\log\log n}) <\infty
\]and by the Borel-Cantelli Lemma we have that with probability 1, only finite many of those events can happen. That is, with probability 1 we have $|X_n^*-E_n^*|<(\varepsilon/2)\sqrt{2\log\log n}$ for all $n$ sufficiently large. Hence, with probability one, for infinitely many $n$ we have:
\[
(1-\varepsilon)\sqrt{2\log \log n}\leq X_n^*\leq (1+\varepsilon)\sqrt{2\log \log n}
\]Hence, we obtain the Law of Iterated Logarithm for Hamilton cycles provided that $k\geq 4$.

\end{proof}

\section{Upper-tail Estimates}\label{upper tail estimates}
In this section we present new upper-tail estimates needed in the proofs of Theorems \ref{thm:main} and \ref{thm:main2}.

\subsection{Proof of Lemma \ref{conc0}}

We denote by $K_{n,n}$ the complete bipartite graph and let $\mathcal P$ denote the set of all perfect matchings in $K_{n,n}$. Clearly, we have

$$| \mathcal P |  = n! . $$

For each $P\in \mathcal P$, let $X_P$ to denote the indicator random variable for the event ``$P$ appears in $ B(n,m)$".
 It is easy to see that

  \begin{equation} \label{formula1}  \mathbb{E}[ X_P] =\frac{(m)_n}{(n^2)_n},  \end{equation}and

\begin{equation} \label{Xexpectation} \mathbb{E}[ X_{n,m}] = n!  \frac{ (m)_n} { (n^2) _n }= n!p_m^n\left( - \frac{1 -p_m}{p_m } +O(1/n) \right)
\end{equation}

\noindent where $p_m:=\frac{m}{n^2}$. For the calculation of equation \eqref{Xexpectation}, see the Appendix. In general, for any fixed bipartite graph
  $H$ with $h$ edges, the probability that  $B(n,m)$ contains $H$ is precisely

$$   \frac{(m)_{h}}{{(n^2)}_{h}}.$$

 Thinking of $H$ as the (simple) graph formed by the union of perfect matchings $P_1, \dots, P_k$, observing that $X_H=X_{P_1}\cdots X_{P_k}$, we obtain that

\begin{equation}   \label{rearranged}
  \mathbb{E} [X_{n,m}^k] = \sum_{P_1,...,P_k\in \mathcal P} \mathbb E [X_{P_1} \dots X_{P_k}] = \sum_{a=0}^{(k-1)n}M(a)\frac{(m)_{kn-a}}{{(n^2)}_{kn-a}},
\end{equation}
where $M(a)$ is the number of (ordered) $k$-tuples
$(P_1,...,P_k)\in \mathcal P^k$, whose union contains exactly $kn-a$ edges. Our main task is to bound $M(a)$ from above.

 Fix $a$ and let $\mathcal L:=\mathcal L(a)$ be the set of all  sequences  $L:= \ell_2,\ldots,\ell_k$   of non-negative integers where  $$ \ell_2 + \dots + \ell_k = a. $$   For each sequence $L = \ell_2, \dots, \ell_k$,  let  $N_{L} $ be the number of $k$-tuples  $(P_1,\ldots,P_k)$ such that
 for every $2 \le t \le k$,  we have  $|P_t \cap (\cup_{j<t}P_j)|=\ell_t$. Clearly, we have

 $$M(a)  =\sum _{ L \in \mathcal L } N_L . $$

 We construct a $k$-tuple in $N_L$ according to the following algorithm:

 \begin{itemize}
\item Let $P_1$ be an arbitrary perfect matching.

\item Suppose that $P_1,\ldots,P_{t-1}$ are given, our aim is to  construct $P_t$. Pick $\ell_t$ edges to be in $P_t \cap \cup_{j=1}^{t-1}P_j$ as follows: first, pick a subset $B_{1,t}$ of $\ell_t$ vertices from the first color class (say $V_1$). Next,
 from each vertex pick an edge which appears in
$\cup_{j=1}^{t-1}P_j$ so that the chosen edges form a matching. Let us denote the obtained partial matching by $E_t$, and observe that $|E_t|=\ell_t$, and that $B_{2,t}:=\left(\cup E_t\right)\cap V_2$ is a set of size $\ell_t$ (where $V_2$ denotes the second color class).

\item  Find a perfect matching $M_t$ between $V_1 \backslash B_{1,t} $ and $V_2 \backslash B_{2,t} $ which has an empty intersection with $\cup_{j=1}^{t-1} P_j$, and set $P_t:= E_t \cup M_t $.
\end{itemize}

Next, we wish to analyze the algorithm. There are $n!$ ways to choose $P_1$. Having chosen $P_1,\ldots,P_{t-1}$, there are ${n \choose {\ell_t} }$ ways to choose $B_{1,t}$.  Each vertex in $B_{1,t}$ has at most $t-1$ different edges in $\cup_{j=1}^{t-1}P_j$. Thus, the number of ways to choose
$E_t$ is at most $(t-1)^{\ell_t} $.  Moreover, once $B_{1,t}$ and $B_{2,t} $ are defined, the number of ways to choose $M_t$ is at most  $(n-\ell_t) !$. This way, we obtain

$$N_L \le n !  \prod_ {t=2}^{k}  {n \choose {\ell_t} } (t-1)^{\ell_t}  (n-\ell_t)! =  n !  \prod_{t=2}^k n!  \frac{(t-1)^{\ell_t} } {\ell_t ! } = (n !)^k  \prod_{t=2}^k   \frac{(t-1)^{\ell_t} } {\ell_t ! }. $$

\noindent By the multinomial identity  and the definition of the set $\mathcal L$,

$$\sum_{L \in \mathcal L}   \prod_{t=2}^k   \frac{(t-1)^{\ell_t} } {\ell_t ! }  =  \frac{1}{a!} (1+ \dots + (k-1)) ^a = \frac{ { k \choose 2}^a}{a!}. $$

\noindent Therefore

\begin{equation} \label{bound1} M(a) = \sum_{L \in \mathcal L} N_L \le   (n!)^k \sum_{ L \in \mathcal L } \prod_{t=2}^k   \frac{(t-1)^{\ell_t} } {\ell_t ! }  = (n!)^k  \frac{ {k \choose 2} ^a}{a!}. \end{equation}

This estimate is sufficient in the case when $a$ is relatively large. However, it is too generous in the case when $a$ is small (the main contribution in LHS of \eqref{rearranged}  comes from this case). In order to sharpen the bound, we refine the estimate on the number of possible $M_t$'s that one can choose in the last step of the algorithm, call this number $\mathcal M_t$ (clearly, $\mathcal M_t$ also depends on the $B_{i,t}$s and we estimate a worse case scenario).
Let $G_t$ be the bipartite graph between $V_1 \backslash B_{1,t} $ and $V_2 \backslash B_{2,t}$ formed by the edges which are not in $\cup_{j=1}^{t-1}P_j$. For each $v \in V_1 \backslash B_{1,t} $,  let $d_v$ be its degree in $G_t$.  By the
Bregman-Minc inequality (see theorem \ref{bregman})

$$\mathcal M_t \le \prod_{ v \in V_1 \backslash B_{1,t} }  (d_v !)^{1/d_v} . $$

\noindent It is clear from the definition that for each $v$

\[
d  := n- \ell_t -(t-1) \le d_v \le n -\ell_t : = D
\]
Call a vertex $v$ {\it good} if $d_v = d $ and {\it bad} otherwise. It is easy to see that $v$ is good if and only if it has exactly $t-1$ different edges in $\cup_{j=1}^{t-1} P_j$ and none of these edges hits  $B_{2,t} $. It follows that the number of good vertices is at least
\[
 n- \ell_t (t-1) - \sum_{j=2}^{t-1} \ell_j  \ge n -a(k-1) -a = n - ka .
\]
Since $(d!)^{1/d} $ is monotone increasing, it follows that
\[
\mathcal M_t \le  (d !) ^{\frac{n-ka}{d} }   (D! )^{ \frac{ ka -\ell_t }{D } }.
\]
Comparing to the previous bound of $(n-\ell_t) ! $, we gain a factor of

\begin{equation} \label{bound2}   \frac { (d !) ^{\frac{n-ka}{d} }   (D! )^{ \frac{ ka -\ell_t }{D } } }{ (n-\ell_t) !  } = \left[  \frac{ (d!)^{1/d} } { (D!)^{1/D } } \right] ^{n-ka}. \end{equation}

\noindent A routine calculation (see Appendix) shows that whenever $ka= o( n)$, the RHS is

\begin{equation}  \label{bound3}  (1 + o(1)) e^{- (t-1) }. \end{equation}

\noindent Thus, for such values of $a$, we have

\begin{equation} \label{bound4} M(a)  \le  (n!)^k  \frac{ {k \choose 2} ^a}{a!}  \prod_{t=2}^k (1+o(1)) e^{-(t-1)}  < 2^k \exp\left(- \frac{k(k-1)}{2}  \right)   (n!)^k  \frac{ {k \choose 2} ^a}{a!} ,  \end{equation}  where the constant 2 can be replaced by any constant larger than 1.
\\
\\
Now  we are ready to bound  $\mathbb E X_{n,m} ^k$. Recall (\ref{rearranged})

$$ \mathbb E X_{n,m}^k =\sum_{a=0}^{(k-1)n}M(a)\frac{(m)_{kn-a}}{{(n^2)}_{kn-a}} . $$

\noindent We split the RHS as

$$ \sum_{a=0}^{T}M(a)\frac{(m)_{kn-a}}{{(n^2)}_{kn-a}}+\sum_{a=T+1 }^{(k-1)n}M(a)\frac{(m)_{kn-a}}{{(n^2)}_{kn-a}} = S_1 +S_2. $$

\noindent where $T= p_mek^2$.  The assumption $k^3 =o(n)$ of the lemma guarantees that $kT =o(n)$. Let  $p_m := \frac{m}{n^2} $. By \eqref{bound4} and lemma \ref{lower factorial} and a routine calculation, we have

$$ S_1= \sum_{a=0}^{T}M(a)\frac{(m)_{kn-a}}{(n^2)_{kn-a}}
\leq \frac{2^k(n!)^kp_m^{nk}}{e^{{k\choose 2}}}\exp\left(-\frac{k^2(1-p_m)}{2p_m}+o(1)\right)\sum_{a=0}^{T}\frac{({k\choose 2})^{a}}{a!}p_m^{-a}. $$

\noindent On the other hand,

$$ \sum_{a=0}^{T}\frac{({k\choose 2})^{a}}{a!}p_m^{-a} <  \sum_{a=0}^{\infty }\frac{({k\choose 2})^{a}}{a!}p_m^{-a}  = e^{ { k \choose2 } /p_m }, $$ so

$$S_1 \le  \frac{2^k  (n!)^kp_m ^{nk}}{e^{{k\choose 2}}}\exp\left(-\frac{k^2(1-p_m)}{2p_m}+o(1)\right)e^{{k\choose 2}/p_m}=
  C_1^k(n!)^k p_m^{nk}, $$

 \noindent where $C_1$ is a constant depending on $p_m$. (In fact we can replace the constant $2$ by any constant larger than 1 in the definition of $C_1$; see the remark following \eqref{bound4}).
  To bound $S_2$, we use \eqref{bound1} and lemma \ref{lower factorial} to obtain

 $$S_2 = \sum_{a >T}  M(a)\frac{(m)_{kn-a}}{(n^2)_{kn-a}}
\leq (n!)^kp_m^{nk}  \exp\left(-\frac{k^2(1-p_m)}{2p_m}+o(1)\right)\sum_{a >T } \frac{({k\choose 2})^{a}}{a!}p_m^{-a}. $$

Notice that we no longer have the term $\frac{2^k}{ e^{{ k \choose 2}}}$. However, as $a$ is large, there is a much better way to bound
$\sum_{a >T } \frac{({k\choose 2})^{a}}{a!}p_m^{-a}. $ Stirling's approximation yields

$$ \sum_{a >T } \frac{({k\choose 2})^{a}}{a!}p_m^{-a} \le \sum_{a >T} \left(\frac{ek^2}{ 2  p_m a } \right)^a < \sum_{a >T} \left(\frac{1}{2} \right)^a = O(1). $$

\noindent It follows that

$$S_2= o( (n!)^k p_m^{nk}) , $$ and thus is negligible for our needs.  Therefore,

$$\mathbb E  [X_{n,m} ^k]= S_1 + S_2 \le C_1^k (n!) p_m^{nk}. $$

\noindent Finally, note that \eqref{Xexpectation} implies

$$(\mathbb E [X_{n,m}])^k = (n!)^k p_m^{nk} \exp \left(\frac{k(1-p_m  )}{p_m } +O(k/n)\right)  \ge C_2^k(n!)^k p_m^{nk}, $$ for an appropiate constant $C_2$. Thus, we get $\mathbb{E}[X_{n,m}^k]/(\mathbb{E}[X_{n,m}]^k)\leq C^k$ by setting $C:=C_1/C_2$.

\subsection{Proof of Lemma \ref{conc1}}
\begin{proof}[Proof of lemma \ref{conc1}]Let $K_n$ be the complete graph of $n$ vertices and denote by $\mathcal H$ the set of Hamilton cycles in $K_n$. Clearly,
\[
|\mathcal H|=\frac{(n-1)!}{2}
\]
For each $H\in \mathcal H$, let $X_H$ denote the indicator random variable for the event ``$H$ appears in $G(n,m)$". It is easy to see that
\[
\mathbb{E}[X_H]=\frac{(m)_n}{{n\choose 2}_n}
\]Thus,
\begin{align}\label{exp x}
\mathbb{E}[X_{n,m}]=N\frac{(m)_n}{{n\choose 2}_n}
\end{align}where above and henceforth we let $N:=(n-1)!/2$. By lemma \ref{lower factorial},
\[
\frac{(m)_n}{{n\choose 2}_n}=p_m^n \exp\left(-\frac{1-p_m}{p_m}+o(1)\right)
\]Hence, calculating the $k$-th moment we obtain:
\begin{align}\label{k moment}
\mathbb{E}[X_{n,m}^k]=\sum_{H_1,\ldots,H_k\in\mathcal H}\mathbb{E}[X_{H_1}\ldots X_{H_k}]=\sum_{a=0}^{(k-1)n}M(a)\frac{(m)_{kn-a}}{{n\choose 2}_{kn-a}}
\end{align}where $M(a)$ is the number of (ordered) $k$-tuples $(H_1,\ldots,H_k)\in \mathcal H^k$. The following lemma gives us bounds for $M(a)$, and it is true for $k\leq \frac{\log n}{8}$.
\begin{lemma}\label{upper bound Ma}For $M(a)$ defined above, if $0\leq a\leq \log^3 n$ we have:
\[
M(a)\leq 3^k N^k\frac{(k(k-1))^a}{e^{k(k-1)}a!}
\]and for $\log^3 n< a\leq (k-1)n$ we have the following weaker bound:
\[
M(a)\leq 3^kN^k\frac{(k(k-1))^a}{a!}
\]
\end{lemma}
\noindent Splitting the sum in \eqref{k moment},
\begin{align}
\mathbb{E}[X_{n,m}^k]&=\sum_{a=0}^{\log^3 n}M(a)\frac{(m)_{kn-a}}{{n\choose 2}_{kn-a}}+\sum_{a=\log^3 n+1}^{(k-1)n}M(a)\frac{(m)_{kn-a}}{{n\choose 2}_{kn-a}}=S_1+S_2
\end{align}allows us to use lemma \eqref{upper bound Ma}. We bound the two sums separately:
\[
S_1=\sum_{a=0}^{\log^3n}M(a)\frac{(m)_{kn-a}}{{n\choose 2}_{kn-a}}\leq \frac{3^kN^kp_m^{nk}}{e^{k(k-1)}}\exp\left(-\frac{k^2(1-p_m)}{p_m}+o(1)\right)\sum_{a=0}^{\log^3n}\frac{(k(k-1))^a}{a!}p_m^{-a}
\]On the other hand,
\[
\sum_{a=0}^{\log^3n}\frac{(k(k-1))^a}{a!}p_m^{-a}\leq \sum_{a=0}^\infty\frac{(k(k-1))^a}{a!}p_m^{-a}=e^{k(k-1)/p_m}
\]so
\[
S_1\leq\frac{3^kN^kp_m^{nk}}{e^{k(k-1)}}\exp\left(-\frac{k^2(1-p_m)}{p_m}+o(1)\right)e^{k(k-1)/p_m}:=C_1^kN^kp_m^{nk}
\]for some appropriate constant $C_1$ (which depends on $k$). To bound $S_2$:
\[
S_2=\sum_{a>{\log^3n}}M(a)\frac{(m)_{kn-a}}{{n\choose 2}_{kn-a}}\leq 3^kN^kp_m^{nk}\exp\left(-\frac{k^2(1-p_m)}{p_m}+o(1)\right)\sum_{a>{\log^3n}}\frac{(k(k-1))^a}{a!}p_m^{-a}
\]However for this case, it is enough to bound the summation using Stirling's approximation, and use $k=O(\log n)$:
\[
\sum_{a>{\log^3n}}\frac{(k(k-1))^a}{a!}p_m^{-a}\leq \sum_{a>{\log^3n}}\left(\frac{k(k-1)e}{p_ma}\right)^a\leq\sum_{a>{\log^3n}}\left(\frac{1}{2}\right)^a=o(1)
\]It follows that
\[
S_2=o(3^kN^kp_m^{nk}),
\]and is thus totally negligible for our needs. Therefore,
\[
\mathbb{E}[X_{n,m}^k]=S_1+S_2\leq C_1^kN^kp_m^{nk}
\]Finally, raising equation (\ref{exp x}) to the $k$-th power yields:
\[
(\mathbb{E}[X_{n,m}])^k=N^kp_m^{nk}\exp\left(-\frac{(1-p_m)k}{p_m}+o(1)\right)\geq C_2^kN^kp_m^{nk}
\]for some constant $C_2$. Hence,
\[
\frac{\mathbb{E}[X_{n,m}^k]}{(\mathbb{E}[X_{n,m}])^k}\leq (C_1/C_2)^k
\]and setting $C:=C_1/C_2$ finishes the proof.
\end{proof}
\begin{proof}[Proof of lemma \ref{upper bound Ma}]Fix $a\leq \log^3 n$, and let $\mathcal L:=\mathcal L(a)$ be the set of all the sequences $L:=(\ell_1,\ldots,\ell_k)$ of non-negative integers where
\[
\ell_2+\ell_3+\ldots+\ell_k=a
\]For each $L=(\ell_2,\ldots,\ell_k)$, let $N_L$ be the number of $k$-tuples $(H_1,\ldots,H_k)$ such that for $2\leq t\leq k$ we have $|H_t\cap (\cup_{i<t}H_i)|=\ell_t$. Clearly we have,
\[
M(a)=\sum_{L\in \mathcal L}N_L
\]we know describe how to construct $k$-tuples in $N_L$.
\begin{enumerate}
\item Pick an arbitrary $H_1$.
\item Assume we are given $H_1,\ldots, H_{t-1}$. Construct a set $E_t$ of edges, of size $\ell_t$ such that $E_t\subset \cup_{i<t}H_i$.
\item Complete $E_t$ into a Hamilton cycle.
\end{enumerate}Next we analyze the algorithm. Clearly there are $N$ ways to perform the first step. For the moment, assume that the number of ways to perform step 2 and 3 (for a fixed $t$) is given by:
\[
3N\frac{(2(t-1))^{\ell_t}}{e^{2(t-1)}\ell_t!}
\]Then, for fixed $L$ we would have the following upper bound on $N_L$:
\[
N_L\leq 3^kN^k\prod_{t=2}\frac{(2(t-1))^{\ell_t}}{e^{2(t-1)}\ell_t!}
\]by the multinomial identity and the definition of the set $\mathcal L$ we have,
\[
\sum_{L\in \mathcal L}\prod_{t=2}\frac{(2(t-1))^{\ell_t}}{e^{2(t-1)}\ell_t!}=\frac{1}{e^{k(k-1)}a!}(2+4+\ldots+2(k-1))^a=\frac{(k(k-1))^a}{e^{k(k-1)}a!}
\]so we obtain the upper bound on $M(a)$,
\[
M(a)\leq 3^kN^k\frac{(k(k-1))^a}{e^{k(k-1)}a!}
\]as claimed. Hence to finish we need to upper bound steps 2-3 of the algorithm.
\\
\\
\textbf{Upper bound on steps 2 and 3.} Assume we are given $H_1,\ldots,H_{t-1}$. For each vertex $v$, consider the set $L(v)$ defined as follows:
\[
L(v):=\{w\mid vw\in (H_1\cup\ldots\cup H_{t-1})\}
\]which we shall refer to as the list of bad vertices of $v$. Note that for each $v$, we have $|L(v)|\leq 2(t-1)$. Pick a subset $V_t\subset V(K_n)$ of size $\ell_t$, say $V_t=\{u_1,\ldots,u_{\ell_t}\}$. We can do so in ${n\choose \ell_t}$ ways. Then, for each $u_i\in V_t$, we select an element, $w_i$, on its list $L(u_i)$. Perform this selection such that if $i\neq j$, then $w_i\neq w_j$. Note that this might not always be possible, in which case the number of ways to perform this step is zero (and we obtain the upper bound trivially). Having chosen the pairs $(u_i,w_i)$, we are going to match them through an edge. Hence, we have at most
\[
{n\choose \ell_t}(2(t-1))^{\ell_t}
\]number of ways to construct $E_t$. Now our task is to upper bound the number of ways we can complete $E_t$ into a Hamilton cycle without using any edges in $\cup_{i<t}H_i$.
\\
\\
First, we are going to collapse the edges in $E_t$ into vertices, and identify them by $w_i$. Hence, we now have $V(K_n)\backslash V_t$ as vertex set (that is, $n-\ell_t$ vertices). We are going to upper bound a bigger quantity: The number of \textbf{oriented} Hamilton cycles, such that for no vertex $v$, we have $v\rightarrow w$ for some $w\in L(v)$, which henceforth we shall refer to as ``$v$ is bad".
\\
\\
Let $N(t)$ be the quantity we wish to upper bound (that is, the number of oriented Hamilton cycles with no bad vertices). Hence,
\begin{align}
N(t)&=(n-\ell_t-1)!-\sum_{v_1}\#\{H\mid v_1 \text{ bad in }H\}+\sum_{v_1,v_2}\#\{H\mid v_1,v_2 \text{ bad in }H\}-\dots\nonumber\\
=&s_0-s_1+s_2-\cdots
\end{align}where $s_i=\sum_{v_1,\ldots,v_i}\#\{H\mid v_1\ldots,v_i \text{ bad in } H\}$. We now give upper and lower bounds on $s_i$, and we also argue why it is enough to consider the terms up to $i=\log^2 n$:
\\
\\
\textbf{Upper bound on $s_{t}$}: First we choose the $i$ vertices that will be bad. There are ${n-\ell_t\choose i}$ ways to do so. Say we chose $\{v_1,\ldots,v_i\}$. Then there are at most $2(t-1)$ many ways to make each vertex bad, hence a total of at most $(2(t-1))^i$ ways to make $v_r$ bad ($1\leq r\leq i$). Hence, we have $v_r\rightarrow x_r$ for some $x_r$ in its set $L(v_r)$. Collapse $v_r$ and $x_r$ onto a single vertex (for $1\leq r\leq i$), so now we have $n-\ell_t-i$ vertices. Then form any oriented Hamilton cycle on these vertices, so we have $(n-\ell_t-i-1)!$ ways to do so (then uncollapse them to obtain an oriented Hamilton cycles on $n-\ell_t$ vertices). Hence,
\begin{align}
s_i&\leq {n-\ell_t\choose i}(2(t-1))^i(n-\ell_t-i-1)!\nonumber\\
&=\frac{(n-\ell_t)!}{n-\ell_t-i}\cdot \frac{(2(t-1))^i}{i!}\nonumber\\
&=\frac{n-\ell_t}{n-\ell_t-i}\cdot (n-\ell_t-1)!\cdot \frac{(2(t-1))^i}{i!}\nonumber\\
&=\left(1+O\left(\frac{\ell_t+i}{n}\right)\right)(n-\ell_t-1)!\cdot \frac{(2(t-1))^i}{i!}
\end{align}but since we are considering $i\leq \log^2 n$ and $\ell_t\leq a\leq \log^3 n$ we have: $$s_i\leq (1+O(\log ^3 n/n))(n-\ell_t-1)!\cdot \frac{(2(t-1))^i}{i!}$$
\\
\\
\textbf{Truncation:} We show that $|\sum_{i=\log^2 n}^{n-1} (-1)^i s_i|$ is small. Indeed,
\begin{align}
\left|\sum_{i=\log^2 n}^{n-1} (-1)^i s_i\right|&\leq \sum_{i=\log^2 n}^{n-1}\frac{n-\ell_t}{n-\ell_t-i}(n-\ell_t-1)!\cdot \frac{(2(t-1))^i}{i!}\nonumber\\
&\leq (n-\ell_t-1)! \sum_{i=\log^2 n}^{n-1}n\frac{(2(t-1))^i}{i!}\nonumber\\
&\leq(n-\ell_t-1)!n^2\frac{(2(t-1))^{\log^2n}}{(\log^2 n)!}\nonumber\\
&=(n-\ell_t-1)!o(e^{-2(t-1)}/n)
\end{align}where the second to last inequality holds since the summands are in decreasing order (as $t$ is at most $\frac{\log n}{8}$).
\\
\\
\textbf{Lower bound on $s_t$}: For this bound, we are only going to consider $\{v_1,\ldots,v_i\}$ such that their lists are disjoint. Intuitively, almost all ${n-\ell_t\choose i}$ options are good since the sizes of the lists are of order $t$ (which will be logarithmic). Let $\alpha_i$ be the number of $\{v_1,\ldots,v_i\}$ such that $L(v_t)\cap L(v_r)=\emptyset$ for $t\neq r$ and $|L(v_r)|=2(t-1)$. Hence,
\begin{align}
s_i&\geq \alpha_i (2(t-1))^i(n-\ell_t-i-1)!\nonumber\\
&=\left(\frac{\alpha_i}{{n-\ell_t\choose i}}\right){n-\ell_t\choose i}(2(t-1))^i(n-\ell_t-i-1)!\nonumber\\
&=\left(\frac{\alpha_i}{{n-\ell_t\choose i}}\right)\left(1+O\left(\frac{i+\ell_t}{n}\right)\right)(n-\ell_t-1)!\cdot \frac{2(t-1))^i}{i!}
\end{align}Now, we compute $\alpha_i$: First we choose $v_1$ so that $|L(v)|=2(t-1)$. There are $n-\ell_t-O(\log^3 n)$ options for $v_1$. Then, choose $v_2$ so that $|L(v_2)|=2(t-1)$ and $L(v_2)\cap L(v_1)$ is empty. There are at most $(2(t-1))^2$ many vertices, $u$, such that $L(u)\cap L(v_1)$ is not empty (to see this note that $L(v_1)$ has size $(2(t-1))$ and each member of $L(v_1)$ is in at most $(2(t-1))$ many lists). Hence, the number of ways to pick $v_2$ is at least $n-\ell_t-O(\log^3 n)-(2(t-1))^2$. Continue in the manner to obtain (after dividing by the $i!$ that comes from double counting) the following lower bound:
\begin{align}
\alpha&\geq \frac{(n-\ell_t-O(\log^3 n))(n-\ell_t-O(\log^3 n)-(2(t-1))^2)\cdots (n-\ell_t-O(\log^3 n)-(i-1)(2(t-1))^2)}{i!}\nonumber\\
&\geq \frac{(n-\ell_t-O(\log^4 n ))^i}{i!}
\end{align}where the last inequality uses $i\leq \log^2n$ and $t\leq (\log n)/8$. We compare with ${n-\ell_t\choose i}$ as follows:
\begin{align}
\frac{\alpha_i}{{n-\ell_t\choose i}}&\geq\frac{\frac{(n-\ell_t-O(\log^4 n ))^i}{i!}}{\frac{(n-\ell_t)_i}{i!}}\nonumber\\
&=\frac{(n-\ell_t-O(\log^4 n ))^i}{(n-\ell_t)_i}\nonumber\\
&=\frac{(n-\ell_t)^i(1-O(\log^4 n/n))^i}{(n-\ell_t)^i(1+O(i^2/n))}\nonumber\\
&=(1-O(\log^{6}n/n))
\end{align}where above we use $(n-\ell_t)_i=(n-\ell_t)^i(1+O(i^2/n))$ which is valid for $i\leq \log^2 n$. Hence, putting everything together we arrive at the lower bound:
\[
s_i\geq (1-O(\log^{6}n /n))(n-\ell_t-1)!\cdot \frac{2(t-1))^i}{i!}
\]Hence, we have that for all $i\leq \log^2 n$ the following bounds on $s_i$:
\[
(1-O(\log^{6}n /n))(n-\ell_t-1)!\cdot \frac{2(t-1))^i}{i!}\leq s_i\leq (1+O(\log^{6}n /n))(n-\ell_t-1)!\cdot \frac{2(t-1))^i}{i!}
\]which implies:
\begin{align}
\sum_{i=0}^{\log^2 n}(-1)^is_i&\leq\sum_{i=0}^{\log^2n}(n-\ell_t-1)!\frac{(-2(t-1))^i}{i!}(1+(-1)^iO\left(\log^{6}n/n\right))\nonumber\\
&\leq \left(\sum_{i=0}^{\log^2 n}(n-\ell_t-1)!\frac{(-2(t-1))^i}{i!}\right)+\left(\sum_{i=0}^{\log^2 n}(n-\ell_t-1)!\frac{(2(t-1))^i}{i!}O\left(\frac{\log^{6}n}{n}\right)\right)\nonumber\\
&\leq (n-\ell_t-1)!\left(e^{-2(t-1)}(1+o(1))+e^{2(t-1)}\cdot O\left(\frac{\log^6n}{n}\right)\right)\nonumber\\
&=(n-\ell_t-1)!e^{-2(t-1)}\left((1+o(1)+O\left(\frac{e^{4(t-1)}\log^6 n}{n}\right)\right)\nonumber\\
&=(n-\ell_t-1)!e^{-2(t-1)}(1+o(1))\nonumber
\end{align}where the last equality uses the fact that $t\leq k\leq \frac{\log n}{8}$. Putting everything together we have:
\begin{align}
\sum_{i=0}^{n-1}(-1)^is_i&=\sum_{i=0}^{\log^2n}(-1)^is_i+\sum_{i=\log^2n}^{n-1}(-1)^is_i\nonumber\\
&\leq \sum_{i=0}^{\log^2n}(-1)^is_i+(n-\ell_t-1)!o(e^{-2(t-1)}/n)\nonumber\\
&\leq (n-\ell_t-1)!e^{2(t-1)}(1+o(1))\nonumber
\end{align}Thus, the number of ways to complete $E_t$ into a Hamilton cycles is upper bounded by:
\[
(n-1-\ell_t)!e^{-2(t-1)}(1+o(1))
\]Putting it together with the upper bound on the number of ways to construct $E_t$ we obtain that the upper bound on Steps 2 and 3 of our algorithm is given by:
\begin{align}
(1+o(1))(n-1-\ell_t)!e^{-2(t-1)}{n\choose \ell_t}(2(t-1))^{\ell_t}=(1+o(1))2N\frac{(2(t-1))^{\ell_t}}{e^{2(t-1)}}\leq 3N\frac{(2(t-1))^{\ell_t}}{e^{2(t-1)}}\nonumber
\end{align}

\end{proof}

\section{Appendix}
\textbf{Proof of lemma \ref{lower factorial}:}
Let $t, \ell$ be such that $\ell=o(t^{2/3})$. Then,
\begin{align}
(t)_\ell&=t(t-1)\cdots (t-\ell+1) \nonumber\\
&=t^{\ell}\prod_{i=0}^{\ell-1}(1-i/t)\nonumber\\
&=t^{\ell}\prod_{i=0}^{\ell-1}e^{-i/t+O(i^2/t^2)}\nonumber\\
&=t^{\ell}\exp\left(\sum_{i=0}^{\ell-1}-i/t+O(i^2/t^2)\right)\nonumber\\
&=t^{\ell}\exp\left(-\frac{\ell(\ell-1)}{2t}+O(\ell^3/t^2)\right)\nonumber\\
&=t^{\ell}\exp\left(-\frac{\ell(\ell-1)}{2t}+o(1)\right)\nonumber
\end{align}as claimed.\\
\\
\textbf{Approximation of expected value (Perfect matchings):} For a subgraph $H$ of $K_{n,n}$ with exactly $h$ edges, the probability that $H$ appears in $B(n,m)$ is exactly:
$$
\frac{{n^2-h \choose m-h}}{{n^2 \choose m}}=\frac{(m)_h}{(n^2)_h}
$$Let $H$ be a perfect matching on $K_{n,n}$, then $h=n$, so we can apply Lemma \ref{lower factorial} to obtain:
\begin{align*}
\frac{(m)_n}{(n^2)_n}&=\frac{ m^n \exp\left(-\frac{n(n-1)}{2m}+O(1/n)\right)}{(n^2)^n\exp\left(-\frac{n(n-1)}{2n^2}+O(1/n)\right)}\\
&=\frac{m^n}{(n^2)^n}\exp\left(-\frac{n^2}{2m}+\frac{1}{2}+O(1/n)\right)\\
&=p_m^n\exp\left(-\frac{1-p_m}{2p_m}+O(1/n)\right)
\end{align*}where in the last equality we used $p_m:=m/n^2$. Since there are a total of $n!$ perfect matchings, we obtain by linearity:
\[
\mathbb{E}[X_{n,m}]=n!p_m^n\exp\left(-\frac{1-p_m}{2p_m}+O(1/n)\right)
\]
\textbf{Approximation of expected value (Hamilton cycles):} Just like above, let $H$ be a hamilton cycle in $K_n$. Then the probability that $H$ appears in $G(n,m)$ is given by:
\begin{align*}
\frac{{{n\choose 2}-n \choose m-n}        }{     {{n\choose 2}\choose m      }}&=\frac{(m)_n}{{n\choose 2}_n}\\
&=\frac{m^n}{{n\choose 2}^n}\exp\left(-\frac{n^2}{2m}+\frac{n^2}{2{n\choose 2}}+O(1/n)\right)\\
&=p_m^n\exp\left(-\frac{1-p_m}{p_m}+O(1/n)\right)
\end{align*}by linearity, one obtains the desired approximation.
\\
\\
\textbf{Computation of equation \eqref{bound2}}: We are going to use the following upper and lower bounds for the factorial:
\[
\sqrt{2\pi s}(s/e)^s\leq s!\leq \sqrt{2\pi s}(s/e)^se^{1/12s}
\]Hence,
\begin{align}
\left[  \frac{ (d!)^{1/d} } { (D!)^{1/D } } \right] ^{n-ka}&\leq\left[\frac{(\sqrt{2\pi d}(d/e)^de^{1/12d})^{1/d}}{(\sqrt{2\pi D}(D/e)^D)^{1/D}}\right]^{n-ka}\nonumber\\
&=\left[(1+O(n^{-2}))\frac{d(2\pi d)^{1/2d}}{D(2\pi D)^{1/2D}}\right]^{n-ka}\nonumber\\
&=(1+O(n^{-1})\left[\frac{(2\pi d)^{1/2d}}{(2\pi D)^{1/2D}}\right]^{n-ka}\left[\frac{d}{D}\right]^{n-ka}\nonumber\\
&=(1+o(1))\left[1-\frac{t-1}{n-\ell_t}\right]^{n-ka}\nonumber\\
&=(1+o(1))e^{t-1}\nonumber	
\end{align}as desired.  (Here we use the assumption that  $ka =o(n)$.)
\section{Acknowledgments}
We would like to thank Kyle Luh for his useful comments during the draft of this paper.


\begin{thebibliography}{9}
\bibitem{AlonSpencer}
N. Alon and J. H. Spencer. \textbf{The Probabilistic Method}.
Vol. 7, John Wiley \& Sons, 2011.


\bibitem{BarbourKaronskiRucinski} A. Barbour, M. Karo\'{n}ski, and A. Ruci\'{n}ski. \emph{A central limit theorem for decomposable random variables with applications to random graphs}. Journal of Combinatorial
Theory, Series B, 47(2) (1989), 125–-145.



\bibitem{Bregman}  L. M.  Bregman,  \emph {Certain properties of nonnegative matrices and their permanents,}   (Russian) Dokl. Akad. Nauk SSSR 211 (1973), 27--30.


\bibitem{DF1} A. Dudek and A. Frieze. \emph{Loose Hamilton Cycles in Random k-Uniform Hypergraphs}.
Electronic Journal of Combinatorics (2011), P48.

\bibitem{DF} A. Dudek and A. Frieze. \emph{Tight Hamilton Cycles in Random Uniform Hypergraphs}.
Random structures and Algorithms 42 (2013), 374--385.

\bibitem{Gao} P. Gao. \emph{Distributions of sparse spanning subgraphs in random graphs}. SIAM Journal on Discrete Mathematics 27.1 (2013), 386--401.

\bibitem{HW} P. Hartman and A. Wintner. \emph{On the law of the iterated logarithm}. American Journal of Mathematics 63.1 (1941), 169--176.

\bibitem{Janson} S. Janson. \emph{ The Numbers of Spanning Trees, Hamilton Cycles and Perfect Matchings in a
Random Graph
}. Combinatorics, Probability and Computing, 3 (1994), 97--126.


\bibitem{JansonLuczakRucinski}
S. Janson, T. {\L}uczak, and A. Rucinski. \textbf{Random graphs}.
Vol. 45, John Wiley \& Sons, 2011.


\bibitem{Kar} M. Karo\'nski. \emph{Balanced subgraphs of large random graphs}. No. 7. UAM, 1984.

\bibitem{KarR} M. Karo\'nski and A. Ruci\'nski. \emph{On the number of strictly balanced subgraphs of a random graph}. In Graph theory (1983), 79--83. Springer Berlin Heidelberg.

\bibitem{Khinchin}  A. Khinchine. \emph{\"{U}ber einen Satz der Wahrscheinlichkeitsrechnung}, Fundamenta Mathematicae 6 (1924), 9--20.

\bibitem{Kolmogorov} A. Kolmogoroff. \emph{\"{U}ber das Gesetz des iterierten Logarithmus}. Mathematische Annalen, 101 (1929), 126--135.

\bibitem{Rinott} Y. Rinott. \emph{On normal approximation rates for certain sums of
dependent random variables.} Journal of Computational and Applied
Mathematics 55.2 (1994), 135--143.

%\bibitem{RW} R. W. Robinson and N. C. Wormald. \emph{Almost all cubic graphs are Hamiltonian}. Random Structures and Algorithms 3.2 (1992), 117--126.
%
%\bibitem{RW1} R. W. Robinson and N. C. Wormald. \emph{Almost all regular graphs are Hamiltonian}. Random Structures and Algorithms 5.2 (1994), 363--374.
\bibitem{Rucinski} A. Ruci\'{n}ski.\emph{When are small subgraphs of a random graph normally distributed?}. Probability Theory and Related Fields, 78(1) (1988), 1–-10.
\end{thebibliography}
\end{document}